\documentclass{amsart} 
\usepackage{amsmath}
\usepackage{amssymb}
\usepackage{pdfsync}
\usepackage{xcolor}
\usepackage[all,2cell,color]{xy} \UseAllTwocells
\xyoption{tips}\SelectTips{cm}{10}
\usepackage{tikz}

\newcommand{\m}[1]{\mathcal{#1}}

\makeatletter
\def\slashedarrowfill@#1#2#3#4#5{%
  $\m@th\thickmuskip0mu\medmuskip\thickmuskip\thinmuskip\thickmuskip
   \relax#5#1\mkern-7mu%
   \cleaders\hbox{$#5\mkern-2mu#2\mkern-2mu$}\hfill
   \mathclap{#3}\mathclap{#2}%
   \cleaders\hbox{$#5\mkern-2mu#2\mkern-2mu$}\hfill
   \mkern-7mu#4$%
}
\def\rightslashedarrowfill@{%
  \slashedarrowfill@\relbar\relbar\mapstochar\rightarrow}
\newcommand\xslashedrightarrow[2][]{%
  \ext@arrow 0055{\rightslashedarrowfill@}{#1}{#2}}
\makeatother
\newcommand{\codP}{\int^b P(b,b)}

\newcommand{\CodP}{\operatorname{Cod}P}

\newcommand{\carrow}{\overset{\cdot\cdot}{\Rightarrow}}

\renewcommand{\bf}{\mathbf}

\theoremstyle{definition}
\newtheorem{definition}{Definition}[section]

\newtheorem{remark}[definition]{Remark}

\theoremstyle{plain}
\newtheorem{proposition}[definition]{Proposition}
\newtheorem{theorem}[definition]{Theorem}
\newtheorem{lemma}[definition]{Lemma}
\newtheorem{corollary}[definition]{Corollary}

\begin{document}

\title{A universal characterisation of codescent objects}

\author{Alexander S. Corner}
\date{30 August 2017}
\address{Department of Engineering and Mathematics, Sheffield Hallam University, Howard Street, Sheffield, S1 1WB}
\email{alex.corner@shu.ac.uk}

\begin{abstract}
In this work we define a 2-dimensional analogue of extranatural transformation and use these to characterise codescent objects. They will be seen as universal objects amongst extrapseudonatural transformations in a similar manner in which coends are universal objects amongst extranatural transformations. Some composition lemmas concerning these transformations are introduced and a Fubini theorem for codescent objects is proven using the universal characterisation description.
\end{abstract}
\maketitle

\section{Introduction}
A slick definition of promonoidal category \cite{Day71} is that it is a pseudomonoid in the bicategory $\bf{Prof}$ of categories, profunctors, and natural transformations. This requires a way of composing two profunctors $F \colon \m{B}^{op} \times \m{A} \rightarrow \m{V}$, $G \colon \m{C}^{op} \times \m{B} \rightarrow \m{V}$, which is given by a coend
    \[
        (G \cdot F)(c,a) = \int^{b \in \m{B}} G(c,b) \times F(b,a).
    \]
Coends are often described as a coequalizer of a particular diagram involving the left and right actions of functors of the form $F \colon \m{C}^{op} \times \m{C} \rightarrow \m{V}$. An equivalent formulation is to define them as universal objects amongst extranatural transformations \cite{EK66}, a way of slightly tweaking the notion of natural transformation. In a similar way as above, they are defined between functors $F \colon \m{C}^{op} \times \m{C} \rightarrow \m{E}$, $G \colon \m{D}^{op} \times \m{D} \rightarrow \m{E}$ and act as a mediator between the left and right actions of each functor. For example, rather than requiring the usual naturality squares to commute, one of the axioms for an extranatural transformation $\alpha \colon F \carrow G$ requires the commutativity of the following diagram.
    \[
        \xy
            (0,0)*{F(b,c)}="a";
            (20,0)*+{F(b,b)}="b";
            (0,-15)*+{F(c,c)}="c";
            (20,-15)*+{G(d,d)}="d";
            {\ar^{F(b,f)} "a" ; "b"};
            {\ar^{\alpha_{bd}} "b" ; "d"};
            {\ar_{F(f,c)} "a" ; "c"};
            {\ar_{\alpha_{cd}} "c" ; "d"};
        \endxy
    \]

The motivation in defining extranatural transformations and using them to characterise codescent objects is to generalise Day's convolution structure to the setting of monoidal bicategories, shown in the author's thesis. These colimits can be seen as a $2$-dimensional generalisation of coends and are defined in reference to pseudofunctors $P \colon \m{A}^{op} \times \m{A} \rightarrow \m{B}$. A slightly weaker notion of codescent object that we consider is the bicodescent object. These objects are a type of bicolimit, having both a $1$-dimensional and $2$-dimensional universal property but requiring only existence and not uniqueness of $1$-cells in the $1$-dimensional part.

In the first section we lay out the definition of extrapseudonatural transformation along with some other basic definitions. The subsequent section describes a host of useful lemmas concerning the composition of such transformations. Following this we characterise bicodescent objects as universal objects amongst extrapseudonatural transformations before finishing with a Fubini theorem for codescent objects.

\subsection*{Notation and conventions}
We will use the following notation for pseudofunctors $F \colon \m{A} \rightarrow \m{B}$ between bicategories. The coherence cells are written
    \[
        \phi^F_{g,f} \colon Fg \cdot Ff \Rightarrow F(g \cdot f)
    \]
and
    \[
        \phi^F_a \colon Fid_a \Rightarrow id_{Fa}.
    \]
Similarly, the coherence cells for bicategories are written
    \[
        r_{f} \colon f \cdot id \Rightarrow f,
    \]
    \[
        l_{f} \colon id \cdot f \Rightarrow f,
    \]
and
    \[
        \alpha_{h,g,f} \colon (h \cdot g) \cdot f \Rightarrow h \cdot (g \cdot f).
    \]
Most coherence cells will not be explicilty written due to lack of space for cumbersome composites, instead being represented by isomorphism symbols. However with these conventions the full diagrams can easily be reproduced. Bicodescent objects will be characterised up to adjoint equivalence \cite{Gur12}.
\subsection*{Acknowledgements}
This work was supported by an EPSRC Studentship at the University of Sheffield and is adapted from parts of the author's PhD thesis. The author thanks their supervisor, Nick Gurski, for their support.

\section{Extrapseudonatural Transformations}\label{sec:cod}
We will define the notion of extrapseudonatural transformation, a weak $2$-dimensional generalisation of extranatural transformations - a similar generalisation is seen in the thesis of \cite{Law15}. We could generalise dinatural transformations in an analogous way, though we do not investigate that here. 

Before discussing extrapseudonatural transformations we will set out a definition of pseudonatural transformation for reference.


\begin{definition}\label{def:psnat}
Let $F, G \colon \m{A} \rightarrow \m{B}$ be pseudofunctors between bicategories. A pseudonatural transformation $\alpha \colon F \Rightarrow G$ consists of
    \begin{itemize}
        \item for each $a \in \m{A}$, a $1$-cell $\alpha_a \colon Fa \rightarrow Ga$ in $\m{B}$;
        \item for each $f \colon a \rightarrow a'$ in $\m{A}$, an invertible $2$-cell
            \[
                \xy
                    (0,0)*{Fa}="a";
                    (24,0)*+{Fa'}="b";
                    (0,-14)*+{Ga}="c";
                    (24,-14)*+{Ga'}="d";
                    {\ar^{Ff} "a" ; "b"};
                    {\ar^{\alpha_{a'}} "b" ; "d"};
                    {\ar_{\alpha_a} "a" ; "c"};
                    {\ar_{Gf} "c" ; "d"};
                    {\ar@{=>}^{\alpha_f} (12,-5) ; (12, -9)};
                \endxy
            \]
        in $\m{B}$.
    \end{itemize}
    These are required to satisfy the following axioms.
        \begin{itemize}
            \item PS1 Given $f \colon a \rightarrow a'$ and $g \colon a' \rightarrow a''$ in $\m{A}$, there is an equality of pasting diagrams
                \[
                    \xy
                        (0,0)*+{Fa}="a";
                        (24,0)*+{Fa'}="b";
                        (48,0)*+{Fa''}="c";
                        (0,-14)*+{Ga}="d";
                        (24,-14)*+{Ga'}="e";
                        (48,-14)*+{G''}="f";
                        {\ar^{Ff} "a" ; "b"};
                        {\ar^{Fg} "b" ; "c"};
                        {\ar_{Gf} "d" ; "e"};
                        {\ar_{Gg} "e" ; "f"};
                        {\ar_{\alpha_a} "a" ; "d"};
                        {\ar|{\alpha_{a'}} "b" ; "e"};
                        {\ar^{\alpha_{a''}} "c" ; "f"};
                        {\ar@/_3.2pc/_{G(gf)} "d" ; "f"};
                        {\ar@{=>}^{\alpha_f} (12,-5) ; (12,-9)};
                        {\ar@{=>}^{\alpha_g} (36,-5) ; (36,-9)};
                        {\ar@{=>}^{\phi^G_{g,f}} (24,-19) ; (24,-23)};
                        (70,-14)*+{Fa}="a1";
                        (94,0)*+{Fa'}="b1";
                        (118,-14)*+{Fa''}="c1";
                        (70,-28)*+{Ga}="d1";
                        (118,-28)*+{G''}="f1";
                        {\ar@/^1pc/^{Ff} "a1" ; "b1"};
                        {\ar@/^1pc/^{Fg} "b1" ; "c1"};
                        {\ar_{G(gf)} "d1" ; "f1"};
                        {\ar_{\alpha_a} "a1" ; "d1"};
                        {\ar^{\alpha_c} "c1" ; "f1"};
                        {\ar|{F(gf)} "a1" ; "c1"};
                        {\ar@{=>}^{\phi^F_{g,f}} (94,-5) ; (94,-9)};
                        {\ar@{=>}^{\alpha_{gf}} (94,-20) ; (94,-24)};
                        {\ar@{=} (57,-14) ; (61,-14)};
                    \endxy
                \]
            in $\m{B}$.

            \item PS2 Given $f, f' \colon a \rightarrow a'$ and $\beta \colon f \Rightarrow f'$ in $\m{A}$, there is an equality of pasting diagrams
                \[
                    \xy
                        (0,0)*{Fa}="a";
                        (24,0)*+{Fa'}="b";
                        (0,-14)*+{Ga}="c";
                        (24,-14)*+{Ga'}="d";
                        {\ar|{Ff'} "a" ; "b"};
                        {\ar^{\alpha_{a'}} "b" ; "d"};
                        {\ar_{\alpha_a} "a" ; "c"};
                        {\ar_{Gf'} "c" ; "d"};
                        {\ar@/^2pc/^{Ff} "a" ; "b"};
                        {\ar@{=>}^{F\beta} (12,6.5) ; (12,2.5)};
                        {\ar@{=>}^{\alpha_{f'}} (12,-5) ; (12,-9)};
                        (46,0)*{Fa}="a1";
                        (70,0)*+{Fa'}="b1";
                        (46,-14)*+{Ga}="c1";
                        (70,-14)*+{Ga'}="d1";
                        {\ar^{Ff} "a1" ; "b1"};
                        {\ar^{\alpha_{a'}} "b1" ; "d1"};
                        {\ar_{\alpha_a} "a1" ; "c1"};
                        {\ar|{Gf} "c1" ; "d1"};
                        {\ar@/_2pc/_{Gf'} "c1" ; "d1"};
                        {\ar@{=>}^{\alpha_f} (58,-5) ; (58,-9)};
                        {\ar@{=>}^{G\beta} (58,-16.5) ; (58,-20.5)};
                        {\ar@{=} (34,-7) ; (38,-7)};
                    \endxy
                \]
            in $\m{B}$.

            \item PS3 Given $a \in \m{A}$, there is an equality of pasting diagrams
                \[
                    \xy
                        (0,0)*{Fa}="a";
                        (24,0)*+{Fa}="b";
                        (0,-14)*+{Ga}="c";
                        (24,-14)*+{Ga}="d";
                        {\ar|{id_{Fa}} "a" ; "b"};
                        {\ar^{\alpha_{a}} "b" ; "d"};
                        {\ar_{\alpha_a} "a" ; "c"};
                        {\ar_{id_{Ga}} "c" ; "d"};
                        {\ar@/^2pc/^{Fid_a} "a" ; "b"};
                        {\ar|{\alpha_a} "a" ; "d"};
                        {\ar@{=>}^{\phi^F_a} (12,6.5) ; (12,2.5)};
                        {\ar@{=>}^{r_{\alpha_a}} (17,-2) ; (15,-6)};
                        {\ar@{=>}_{l^{-1}_{\alpha_a}} (9,-8) ; (7,-12)};
                        (46,0)*{Fa}="a1";
                        (70,0)*+{Fa}="b1";
                        (46,-14)*+{Ga}="c1";
                        (70,-14)*+{Ga}="d1";
                        {\ar^{Fid_a} "a1" ; "b1"};
                        {\ar^{\alpha_{a}} "b1" ; "d1"};
                        {\ar_{\alpha_a} "a1" ; "c1"};
                        {\ar|{Gid_a} "c1" ; "d1"};
                        {\ar@/_2pc/_{id_{Ga}} "c1" ; "d1"};
                        {\ar@{=>}^{\alpha_{id_a}} (58,-5) ; (58,-9)};
                        {\ar@{=>}^{\phi^G_a} (58,-16.5) ; (58,-20.5)};
                        {\ar@{=} (34,-7) ; (38,-7)};
                    \endxy
                \]
            in $\m{B}$.
        \end{itemize}
\end{definition}

Extranatural transformations were first defined by \cite{EK66} for use in their subsequent article on closed categories \cite{EK66b}.

\begin{definition}\label{def:epnat}
Let $P \colon \m{A} \times \m{B}^{op} \times \m{B} \rightarrow \m{D}$ and $Q \colon \m{A} \times \m{C}^{op} \times \m{C} \rightarrow \m{D}$ be pseudofunctors. An \emph{extrapseudonatural transformation} $\beta \colon P \carrow Q$ consists of
    \begin{itemize}
        \item for each $b \in \m{B}$, $c \in \m{C}$, a pseudonatural transformation $\beta_{-bc} \colon P(-,b,b) \Rightarrow Q(-,c,c)$;
        \item for each $g \colon b \rightarrow b'$ in $\m{B}$, an invertible $2$-cell
            \[
                \xy
                    (0,0)*+{P(a,b',b)}="a";
                    (24,0)*+{P(a,b,b)}="b";
                    (0,-14)*+{P(a,b',b')}="c";
                    (24,-14)*+{Q(a,c,c)}="d";
                    {\ar^{P_{1g1}} "a" ; "b"};
                    {\ar_{P_{11g}} "a" ; "c"};
                    {\ar^{\beta_{abc}} "b" ; "d"};
                    {\ar_{\beta_{ab'c}} "c" ; "d"};
                    {\ar@{=>}^{\beta_{agc}} (12,-5) ; (12,-9)};
                \endxy
            \]
        in $\m{D}$;
        \item for each $h \colon c \rightarrow c'$ in $\m{C}$, an invertible $2$-cell
            \[
                \xy
                    (0,0)*+{P(a,b,b)}="a";
                    (24,0)*+{Q(a,c',c')}="b";
                    (0,-14)*+{Q(a,c,c)}="c";
                    (24,-14)*+{Q(a,c,c')}="d";
                    {\ar^{\beta_{abc'}} "a" ; "b"};
                    {\ar_{\beta_{abc}} "a" ; "c"};
                    {\ar^{Q_{1h1}} "b" ; "d"};
                    {\ar_{Q_{11h}} "c" ; "d"};
                    {\ar@{=>}^{\beta_{abh}} (12,-5) ; (12,-9)};
                \endxy
            \]
            in $\m{D}$.
    \end{itemize}
These are required to satisfy the following axioms.
    \begin{itemize}
        \item EP1 Given $f \colon b \rightarrow b'$ and $g \colon b' \rightarrow b''$ in $\m{B}$, there is an equality of pasting diagrams
            \[
                \xy
                    (0,0)*+{P_{ab''b}}="a";
                    (15,-7)*+{P_{ab'b}}="b";
                    (15,-35)*+{P_{abb}}="c";
                    (-15,-7)*+{P_{ab''b'}}="d";
                    (-15,-35)*+{P_{ab''b''}}="e";
                    (0,-42)*+{Q_{acc}}="f";
                    {\ar^{P_{1f1}} "a" ; "b"};
                    {\ar^{P_{1g1}} "b" ; "c"};
                    {\ar^{\beta_{abc}} "c" ; "f"};
                    {\ar_{P_{11g}} "a" ; "d"};
                    {\ar_{P_{11f}} "d" ; "e"};
                    {\ar_{\beta_{ab''c}} "e" ; "f"};
                    {\ar|{P_{11(gf)}} "a" ; "e"};
                    {\ar|{P_{1(gf)1}} "a" ; "c"};
                    %
                    (9,-10)*+{\cong};
                    {\ar@{=>}^{\beta_{a(gf)c}} (0,-24) ; (-4,-27)};
                    (-9,-10)*+{\cong};
                    (55,0)*+{P_{ab''b}}="1";
                    (70,-7)*+{P_{ab'b}}="2";
                    (70,-35)*+{P_{abb}}="3";
                    (40,-7)*+{P_{ab''b'}}="4";
                    (40,-35)*+{P_{ab''b''}}="5";
                    (55,-42)*+{Q_{acc}}="6";
                    (55,-14)*+{P_{ab'b'}}="7";
                    {\ar^{P_{1f1}} "1" ; "2"};
                    {\ar^{P_{1g1}} "2" ; "3"};
                    {\ar^{\beta_{abc}} "3" ; "6"};
                    {\ar_{P_{11g}} "1" ; "4"};
                    {\ar_{P_{11f}} "4" ; "5"};
                    {\ar_{\beta_{ab''c}} "5" ; "6"};
                    {\ar^{P_{11g}} "2" ; "7"};
                    {\ar_{P_{1f1}} "4" ; "7"};
                    {\ar|{\beta_{ab'c}} "7" ; "6"};
                    (55,-7)*+{\cong};
                    {\ar@{=>}^{\beta_{agc}} (62.5,-22) ; (58.5,-25)};
                    {\ar@{=>}^{\beta_{afc}} (47.5,-22) ; (43.5,-25)};
                    {\ar@{=} (25.5,-21) ; (29.5,-21)};
                \endxy
            \]
            for all $a \in \m{A}$, $c \in \m{C}$.
            \item EP2 Given $h \colon c \rightarrow c'$ and $i \colon c' \rightarrow c'$ in $\m{C}$, there is an equality of pasting diagrams
            \[
                \xy
                    (0,0)*+{P_{abb}}="a";
                    (15,-7)*+{Q_{ac''c''}}="b";
                    (15,-35)*+{Q_{ac'c''}}="c";
                    (-15,-7)*+{Q_{acc}}="d";
                    (-15,-35)*+{Q_{acc'}}="e";
                    (0,-42)*+{Q_{acc''}}="f";
                    {\ar^{\beta_{abc''}} "a" ; "b"};
                    {\ar^{Q_{1i1}} "b" ; "c"};
                    {\ar^{Q_{1h1}} "c" ; "f"};
                    {\ar_{\beta_{abc}} "a" ; "d"};
                    {\ar_{Q_{11h}} "d" ; "e"};
                    {\ar_{Q_{11i}} "e" ; "f"};
                    {\ar|{Q_{1(ih)1}} "b" ; "f"};
                    {\ar|{Q_{11(ih)}} "d" ; "f"};
                    %
                    (9,-32)*+{\cong};
                    {\ar@{=>}^{\beta_{ab(ih)}}  (0,-12) ; (-4,-15)};
                    (-9,-32)*+{\cong};
                    (55,0)*+{P_{abb}}="1";
                    (70,-7)*+{Q_{ac''c''}}="2";
                    (70,-35)*+{Q_{ac'c''}}="3";
                    (40,-7)*+{Q_{acc}}="4";
                    (40,-35)*+{Q_{acc'}}="5";
                    (55,-42)*+{Q_{acc''}}="6";
                    (55,-28)*+{Q_{ac'c'}}="7";
                    {\ar^{\beta_{abc''}} "1" ; "2"};
                    {\ar^{Q_{1i1}} "2" ; "3"};
                    {\ar^{Q_{1h1}} "3" ; "6"};
                    {\ar_{\beta_{abc}} "1" ; "4"};
                    {\ar_{Q_{11h}} "4" ; "5"};
                    {\ar|{\beta_{abc'}} "1" ; "7"};
                    {\ar_{Q_{11i}} "5" ; "6"};
                    {\ar^{Q_{11i}} "7" ; "3"};
                    {\ar_{Q_{1h1}} "7" ; "5"};
                    (55,-35)*+{\cong};
                    {\ar@{=>}^{\beta_{abi}} (62.5,-17) ; (58.5,-20)};
                    {\ar@{=>}^{\beta_{abh}} (47.5,-17) ; (43.5,-20)};
                    {\ar@{=} (25.5,-21) ; (29.5,-21)};
                \endxy
            \]
            for all $a \in \m{A}$, $b \in \m{B}$.
            \item EP3 For each $f \colon a \rightarrow a'$ in $\m{A}$ and $g \colon b \rightarrow b'$ in $\m{B}$ there is an equality of pasting diagrams
                    \[
                        \xy
                            (55,0)*+{P_{ab'b}}="a";
                            (70,-7)*+{P_{a'b'b}}="b";
                            (70,-35)*+{P_{a'bb}}="c";
                            (55,-42)*+{Q_{a'cc}}="d";
                            (40,-7)*+{P_{ab'b'}}="e";
                            (40,-35)*+{Q_{acc}}="f";
                            (55,-14)*+{P_{a'b'b'}}="g";
                            {\ar^{P_{f11}} "a" ; "b"};
                            {\ar^{P_{11g}} "b" ; "c"};
                            {\ar^{\beta_{a'bc}} "c" ; "d"};
                            {\ar_{P_{1g1}} "a" ; "e"};
                            {\ar_{\beta_{ab'c}} "e" ; "f"};
                            {\ar_{Q_{f11}} "f" ; "d"};
                            {\ar^{P_{1g1}} "b" ; "g"};
                            {\ar|{\beta_{a'b'c}} "g" ; "d"};
                            {\ar_{P_{f11}} "e" ; "g"};
                            (55,-7)*+{\cong};
                            {\ar@{=>}^{\beta_{a'gc}} (62.5,-22) ; (58.5,-25)};
                            {\ar@{=>}^{\beta_{fbc}} (47.5,-22) ; (43.5,-25)};
                            (0,0)*+{P_{ab'b}}="1";
                            (15,-7)*+{P_{a'b'b}}="2";
                            (15,-35)*+{P_{a'bb}}="3";
                            (0,-42)*+{Q_{a'cc}}="4";
                            (-15,-7)*+{P_{ab'b'}}="5";
                            (-15,-35)*+{Q_{acc}}="6";
                            (0,-28)*+{P_{abb}}="7";
                            {\ar^{P_{f11}} "1" ; "2"};
                            {\ar^{P_{11g}} "2" ; "3"};
                            {\ar^{\beta_{a'bc}} "3" ; "4"};
                            {\ar_{P_{1g1}} "1" ; "5"};
                            {\ar_{\beta_{ab'c}} "5" ; "6"};
                            {\ar_{Q_{f11}} "6" ; "4"};
                            {\ar|{P_{11g}} "1" ; "7"};
                            {\ar^{P_{f11}} "7" ; "3"};
                            {\ar_{\beta_{abc}} "7" ; "6"};
                            (7.5,-18.5)*+{\cong};
                            {\ar@{=>}^{\beta_{agc}} (-7.5,-17) ; (-11.5,-20)};
                            {\ar@{=>}^{\beta_{fb'c}} (1,-33.5) ; (-3,-36.5)};
                            {\ar@{=} (25.5,-21) ; (29.5,-21)};
                        \endxy
                    \]
                for all $c \in \m{C}$.
            \item EP4 For each $f \colon a \rightarrow a'$ in $\m{A}$ and $h \colon c \rightarrow c'$ in $\m{C}$ there is an equality of pasting diagrams
                    \[
                        \xy
                            (0,0)*+{P_{abb}}="a";
                            (15,-7)*+{P_{a'bb}}="b";
                            (15,-35)*+{Q_{a'c'c'}}="c";
                            (0,-42)*+{Q_{a'c'c}}="d";
                            (-15,-7)*+{Q_{acc}}="e";
                            (-15,-35)*+{Q_{acc'}}="f";
                            (0,-14)*+{Q_{a'cc}}="g";
                            {\ar^{P_{f11}} "a" ; "b"};
                            {\ar^{\beta_{a'bc'}} "b" ; "c"};
                            {\ar^{Q_{1h1}} "c" ; "d"};
                            {\ar_{\beta_{abc}} "a" ; "e"};
                            {\ar_{Q_{11h}} "e" ; "f"};
                            {\ar_{Q_{f11}} "f" ; "d"};
                            {\ar^{\beta_{a'bc}} "b" ; "g"};
                            {\ar|{Q_{11h}} "g" ; "d"};
                            {\ar_{Q_{f11}} "e" ; "g"};
                            (55,-35)*+{\cong};
                            {\ar@{=>}^{\beta_{fbc'}} (62.5,-17) ; (58.5,-20)};
                            {\ar@{=>}^{\beta_{abh}} (47.5,-17) ; (43.5,-20)};
                            (55,0)*+{P_{abb}}="1";
                            (70,-7)*+{P_{a'bb}}="2";
                            (70,-35)*+{Q_{a'c'c'}}="3";
                            (55,-42)*+{Q_{a'c'c}}="4";
                            (40,-7)*+{Q_{acc}}="5";
                            (40,-35)*+{Q_{acc'}}="6";
                            (55,-28)*+{Q_{ac'c'}}="7";
                            {\ar^{P_{f11}} "1" ; "2"};
                            {\ar^{\beta_{a'bc'}} "2" ; "3"};
                            {\ar^{Q_{1h1}} "3" ; "4"};
                            {\ar_{\beta_{abc}} "1" ; "5"};
                            {\ar_{Q_{11h}} "5" ; "6"};
                            {\ar_{Q_{f11}} "6" ; "4"};
                            {\ar|{\beta_{abc}} "1" ; "7"};
                            {\ar^{Q_{f11}} "7" ; "3"};
                            {\ar_{Q_{1h1}} "7" ; "6"};
                            (-7.5,-23.5)*+{\cong};
                            {\ar@{=>}^{\beta_{fbc}} (1,-4.5) ; (-3,-8.5)};
                            {\ar@{=>}^{\beta_{a'bh}} (7.5,-22) ; (3.5,-25)};
                            {\ar@{=} (25.5,-21) ; (29.5,-21)};
                        \endxy
                    \]
                for all $b \in \m{B}$.
            \item EP5 For each $a \in \m{A}$, $b \in \m{B}$, and $c \in \m{C}$,
                \[
                    \beta_{a1_bc} = id_{\beta_{abc} \cdot P_{abb}}, \beta_{ab1_c} = id_{Q_{acc} \cdot \beta_{abc}}.
                \]
            \item EP6 Given $g, g' \colon b \rightarrow b'$ and $\gamma \colon g \Rightarrow g'$ in $\m{B}$, there is an equality of pasting diagrams
            \[
                \xy
                    (0,0)*+{P_{ab'b}}="a";
                    (30,0)*+{P_{abb}}="b";
                    (0,-18)*+{P_{ab'b'}}="c";
                    (30,-18)*+{Q_{acc}}="d";
                    {\ar_{P_{ag'b}} "a" ; "b"};
                    {\ar_{P_{ab'g'}} "a" ; "c"};
                    {\ar^{\beta_{abc}} "b" ; "d"};
                    {\ar_{\beta_{ab'c}} "c" ; "d"};
                    {\ar@/^2pc/^{P_{agb}} "a" ; "b"};
                    {\ar@{=>}^{\beta_{ag'c}} (12,-7) ; (12,-11)};
                    {\ar@{=>}^{P_{a\gamma b}} (12, 6) ; (12, 2)};
                    (70,0)*+{P_{ab'b}}="1";
                    (100,0)*+{P_{abb}}="2";
                    (70,-18)*+{P_{ab'b'}}="3";
                    (100,-18)*+{Q_{acc}}="4";
                    {\ar^{P_{agb}} "1" ; "2"};
                    {\ar^{P_{ab'g}} "1" ; "3"};
                    {\ar^{\beta_{abc}} "2" ; "4"};
                    {\ar_{\beta_{ab'c}} "3" ; "4"};
                    {\ar@/_2pc/_{P_{ab'g'}} "1" ; "3"};
                    {\ar@{=>}^{\beta_{agc}} (82,-7) ; (82,-11)};
                    {\ar@{=>}_{P_{ab'\gamma}} (68,-10) ; (64,-10)};
                    {\ar@{=} (45,-9) ; (48,-9)};
                \endxy
            \]
        for all $a \in \m{A}$, $c \in \m{C}$.
        \item EP7 Given $h, h' \colon c \rightarrow c'$ and $\delta \colon h \Rightarrow h'$ in $\m{C}$, there is an equality of pasting diagrams
                     \[
                \xy
                    (0,0)*+{P_{abb}}="a";
                    (30,0)*+{Q_{ac'c'}}="b";
                    (0,-18)*+{Q_{acc}}="c";
                    (30,-18)*+{Q_{acc'}}="d";
                    {\ar^{\beta_{abc'}} "a" ; "b"};
                    {\ar_{\beta_{abc}} "a" ; "c"};
                    {\ar^{Q_{ahc'}} "b" ; "d"};
                    {\ar^{Q_{ach}} "c" ; "d"};
                    {\ar@/_2pc/_{Q_{ah'c}} "c" ; "d"};
                    {\ar@{=>}^{\beta_{abh}} (12,-7) ; (12,-11)};
                    {\ar@{=>}^{Q_{ac\delta}} (12, -20) ; (12, -24)};
                    (60,0)*+{P_{abb}}="1";
                    (90,0)*+{Q_{ac'c'}}="2";
                    (60,-18)*+{Q_{acc}}="3";
                    (90,-18)*+{Q_{acc'}}="4";
                    {\ar^{\beta_{abc'}} "1" ; "2"};
                    {\ar_{\beta_{abc}} "1" ; "3"};
                    {\ar_{Q_{ah'c'}} "2" ; "4"};
                    {\ar_{Q_{ach'}} "3" ; "4"};
                    {\ar@/^2pc/^{Qahc'} "2" ; "4"};
                    {\ar@{=>}^{\beta_{abh'}} (72,-7) ; (72,-11)};
                    {\ar@{=>}_{Q_{a\delta c'}} (96,-9) ; (92,-9)};
                    {\ar@{=} (45,-9) ; (48,-9)};
                \endxy
            \]
        for all $a \in \m{A}$, $b \in \m{B}$.
    \end{itemize}
\end{definition}

\begin{remark}\label{constant}
There are a large number of axioms in the definition of an extrapseudonatural transformation. In practice we will find that only a subset of these need to be checked. For example, if $Q$ is a constant pseudofunctor in the definition above then each $2$-cell $\beta_{abh}$ is in fact an identity $id_{Q} \cdot id_{\beta_{ab\cdot}}$. This then means that EP2, EP4, EP7, and the second part of EP5 all hold automatically. In many cases from this point onwards we deal with extrapseudonatural transformations into or out of constant pseudofunctors on an object of a bicategory. We denote the constant pseudofunctor on an object $H$ either by $\Delta_H \colon \bf{1}^{op} \times \bf{1} \rightarrow \m{C}$ or, most commonly, simply by $H$.
\end{remark}

In the one dimensional case a transformation is extranatural in the pair $(a,b)$ if and only if it is extranatural in $a$ and $b$ separately. For extrapseudonatural transformations, however, this is no longer true. Being extrapseudonatural in $a$ and $b$ separately only implies extrapseudonaturality in $(a,b)$ under the conditions of the following lemma, though the usual converse still holds.

\begin{lemma}\label{compatibility}
Let $P \colon \m{A}^{op} \times \m{B}^{op} \times \m{A} \times \m{B} \rightarrow \m{C}$ be a pseudofunctor and let $X \in \m{C}$. Suppose, respectively, that for fixed $a \in \m{A}$ and for fixed $b \in \m{B}$, $\gamma_{a-} \colon P(a,-,a,-) \overset{\cdot\cdot}{\Rightarrow} X$ and $\gamma_{-b} \colon P(-,b,-,b) \overset{\cdot\cdot}{\Rightarrow} X$ are extrapseudonatural transformations such that $(\gamma_{a-})_b = (\gamma_{-b})_a$. If there is an equality of pasting diagrams
    \[
        \xy
            (0,0)*+{P_{a'b'ab}}="a";
            (25,-10)*+{P_{a'bab}}="b";
            (50,-10)*+{P_{abab}}="c";
            (10,-25)*+{P_{a'b'a'b}}="d";
            (10,-50)*+{P_{a'b'a'b'}}="e";
            (50,-50)*+{X}="f";
            (35,-35)*+{P_{a'ba'b}}="g";
            {\ar@/^1.5pc/^{P_{fgab}} "a" ; "c"};
            {\ar|{P_{a'gab}} "a" ; "b"};
            {\ar|{P_{fbab}} "b" ; "c"};
            {\ar@/_2pc/_{P_{a'b'fg}} "a" ; "e"};
            {\ar|{P_{a'b'fb}} "a" ; "d"};
            {\ar|{P_{a'b'a'g}} "d" ; "e"};
            {\ar^{\gamma_{ab}} "c" ; "f"};
            {\ar_{\gamma_{a'b'}} "e" ; "f"};
            {\ar|{\gamma_{a'b}} "g" ; "f"};
            {\ar|{P_{a'ga'b}} "d" ; "g"};
            {\ar|{P_{a'bfb}} "b" ; "g"};
            (75,0)*+{P_{a'b'ab}}="a";
            (100,-10)*+{P_{ab'ab}}="b";
            (125,-10)*+{P_{abab}}="c";
            (85,-25)*+{P_{a'b'ab'}}="d";
            (85,-50)*+{P_{a'b'a'b'}}="e";
            (125,-50)*+{X}="f";
            (110,-35)*+{P_{ab'ab'}}="g";
            {\ar@/^1.5pc/^{P_{fgab}} "a" ; "c"};
            {\ar|{P_{fb'ab}} "a" ; "b"};
            {\ar|{P_{agab}} "b" ; "c"};
            {\ar@/_2pc/_{P_{a'b'fg}} "a" ; "e"};
            {\ar|{P_{a'b'ag}} "a" ; "d"};
            {\ar|{P_{a'b'fb'}} "d" ; "e"};
            {\ar^{\gamma_{ab}} "c" ; "f"};
            {\ar_{\gamma_{a'b'}} "e" ; "f"};
            {\ar|{\gamma_{ab'}} "g" ; "f"};
            {\ar|{P_{fb'ab'}} "d" ; "g"};
            {\ar|{P_{ab'ag}} "b" ; "g"};
            {\ar@{=>}^{\gamma_{fb}} (40,-20) ; (40,-24)};
            {\ar@{=>}^{\gamma_{ag}} (115,-20) ; (115,-24)};
            {\ar@{=>}^{\gamma_{a'g}} (25,-40) ; (25,-44)};
            {\ar@{=>}^{\gamma_{fb'}} (100,-40) ; (100,-44)};
            (17.5,-17.5)*+{\cong};
            (92.5,-17.5)*+{\cong};
            (25,-5)*+{\cong};
            (100,-5)*+{\cong};
            (0,-25)*+{\cong};
            (75,-25)*+{\cong};
            (62.5,-15)*+{=};
        \endxy
    \]
then these $2$-cells constitute an extrapseudonatural transformation $\gamma \colon P \overset{\cdot\cdot}{\Rightarrow} X$.
\end{lemma}
\begin{proof}
All of the axioms to check extrapseudonaturality for the above $2$-cells are satisfied as a result of the corresponding axioms for the individual transformations, naturality of coherence cells for $P$ and of those in $\m{A}$ and $\m{B}$, as well as the equality of $2$-cells stated above which is needed for axioms EP1 and EP2.
\end{proof}

\begin{remark}
The statement labelled above as a proof might seem insufficient to be described as such. The equality of pasting diagrams above is essentially the single obstruction to extrapseudonaturality in each variable implying extrapseudonaturality in the pair. The manipulation of the definitions is relatively simple, though fairly cumbersome, and one finds that, in trying to show extrapseudonaturality in the pair, the above equality is the only thing standing in the way of the implication.
\end{remark}
\begin{lemma}\label{fixed}
Let $P \colon \m{A}^{op} \times \m{B}^{op} \times \m{A} \times \m{B} \rightarrow \m{C}$ be a pseudofunctor and let $X \in \m{C}$. Suppose that $\gamma \colon P \carrow X$ is an extrapseudonatural transformation. If $a \in \m{A}$ is fixed then there is an extrapseudonatural transformation $\gamma_{a-} \colon P(a,-a,-) \carrow X$. Similarly, if $b \in \m{B}$ is fixed then there is an extrapseudonatural transformation $\gamma_{-b} \colon P(-,b,-,b) \carrow X$.
\end{lemma}
\begin{proof}
Since the pseudofunctor in the codomain is constant at the object $X$ we need only check that EP1, EP3, EP6, and the first part of EP5 hold. The axiom EP3 is a consequence of the pseudofunctor axioms for $P$, the first part of EP5 holds for $\gamma_{a-}$ via the axiom EP5 for $\gamma$, while EP6 holds for $\gamma_{a-}$ via EP6 for $\gamma$. To show that EP1 holds for $\gamma_{a-}$ is more involved but still fairly simple, resulting from instances of EP1 and EP6 for $\gamma$. We must show that two diagrams are equal, one featuring $\gamma_{1_a,hg}$ and the other featuring both $\gamma_{1_{a},g}$ and $\gamma_{1_{a},h}$. By EP1 for $\gamma$ the diagram featuring $\gamma_{1_{a},g}$ and $\gamma_{1_{a},h}$ is equal to one featuring $\gamma_{1_a 1_a,hg}$. At this point we use EP6 to show that this diagram is equal to the diagram featuring $\gamma_{1_a,hg}$. Hence $\gamma_{a-}$ is an extrapseudonatural transformation. A similar argument holds for $\gamma_{-b}$.
\end{proof}

\section{Composition Lemmas}
The article of \cite{EK66b} in which extranatural transformations are defined also investigates the ways in which they can be composed. We now present generalisations of the simplest forms of these arguments for extrapseudonatural transformations.

We can picture the first of the composition lemmas in a string diagram format. Pseudoaturality is presented as a straight line between bicategories while extrapseudonaturality is a cup or cap between a bicategory and its opposite. In this lemma we have a pseudonatural transformation $\beta \colon F \Rightarrow G$ where $F,G \colon \m{A}^{op} \times \m{A} \rightarrow \m{C}$, along with an extrapseudonatural transformation $\gamma \colon G \carrow H$, for some object $H \in \m{C}$. We define a composite which results in an extrapseudonatural transformation $F \carrow H$. Graphically this looks like the following diagram, sometimes referred to as `stalactites' due to the shape.

\begin{center}

\begin{tikzpicture}[y=0.80pt, x=0.80pt, yscale=-2.500000, xscale=2.500000, inner sep=0pt, outer sep=0pt]
\path[draw=black,line join=miter,line cap=butt,even odd rule,line width=0.800pt]
  (440.0000,517.3622) .. controls (440.0000,517.3622) and (440.0378,527.4384) ..
  (449.8967,527.4384) .. controls (459.9800,527.4384) and (459.9056,517.4064) ..
  (459.9056,517.4064);
  \path[draw=black,line join=miter, line cap=butt, even odd rule, line width=0.800pt] (440.0000,513) -- (440.0000,483);
  \path[draw=black,line join=miter, line cap=butt, even odd rule, line width=0.800pt] (459.9056,513) -- (459.9056,483);
  \path[draw=black,line join=miter, line cap=butt, even odd rule, line width=0.800pt] (465,500) -- (469,500);
    \path[draw=black,line join=miter, line cap=butt, even odd rule, line width=0.800pt] (465,501) -- (469,501);
    \path[draw=black,line join=miter,line cap=butt,even odd rule,line width=0.800pt]
  (472.0000,495.3622) .. controls (472.0000,495.3622) and (472.0378,505.4384) ..
  (481.8967,505.4384) .. controls (491.9800,505.4384) and (491.9056,495.4064) ..
  (491.9056,495.4064);
%
\node() at (442,515){$\m{A}^{op}$};
\node() at (442,480.7){$\m{A}^{op}$};
\node() at (474,493){$\m{A}^{op}$};
\node() at (460,515){$\m{A}$};
\node() at (492,493){$\m{A}$};
\node() at (460,480.7){$\m{A}$};

\end{tikzpicture}
\end{center}

\begin{lemma}\label{comp1}
Let $F$, $G \colon \m{A}^{op} \times \m{A} \rightarrow \m{C}$ be pseudofunctors and let $H \in \m{C}$. Suppose that $\beta \colon F \Rightarrow G$ is a pseudonatural transformation and that $\gamma \colon G \carrow H$ is an extrapseudonatural transformation. Then there is an extrapseudonatural transformation from $F$ to $H$ given by composites of the cells constituting $\beta$ and $\gamma$.
\end{lemma}
\begin{proof}
The $1$-cells of the extrapseudonatural transformation are given by the composites $\delta_a = \gamma_a \cdot \beta_{aa}$. Given a $1$-cell $f \colon a \rightarrow b$ in $\m{A}$, we give $2$-cells $\delta_f \colon \delta_a \cdot F(f,1) \Rightarrow \delta_b \cdot F(1,f)$ by the following diagram.
                    \[
                        \xy
                            (0,-5)*+{F(b,a)}="1";
                            (15,-12)*+{F(a,a)}="2";
                            (15,-35)*+{G(b,a)}="3";
                            (0,-42)*+{H}="4";
                            (-15,-12)*+{F(b,b)}="5";
                            (-15,-35)*+{G(b,b)}="6";
                            (0,-28)*+{G(a,a)}="7";
                            {\ar^{F(f,1)} "1" ; "2"};
                            {\ar^{\beta_{aa}} "2" ; "3"};
                            {\ar^{\gamma_a} "3" ; "4"};
                            {\ar_{F(1,f)} "1" ; "5"};
                            {\ar_{\beta_{bb}} "5" ; "6"};
                            {\ar_{\gamma_b} "6" ; "4"};
                            {\ar|{\beta_{ba}} "1" ; "7"};
                            {\ar^{G(f,1)} "7" ; "3"};
                            {\ar_{G(1,f)} "7" ; "6"};
                            {\ar@{=>}^{\beta_{f1}} (9,-17) ; (5,-20)};
                            {\ar@{=>}^{\beta_{1f}^{-1}} (-7.5,-17) ; (-11.5,-20)};
                            {\ar@{=>}^{\gamma_f} (1,-33.5) ; (-3,-36.5)};
                            %
                        \endxy
                    \]
As per Remark \ref{constant}, $\Delta_H$ is a constant pseudofunctor and so the other $2$-cells required are all identities.

The axioms EP2-5 are simple to check, whilst EP1 requires a sequence of involved pasting diagrams. We will consider the initial and final pasting diagrams in the sequence and describe the steps required to complete the proof. The left-hand diagram of axiom EP1 is given, in this instance, by the following pasting diagram.
    \[
        \xy
            (0,0)*+{F(c,a)}="1";
            (30,0)*+{F(b,a)}="2";
            (60,0)*+{F(a,a)}="3";
            (90,-15)*+{G(a,a)}="4";
            (90,-45)*+{H}="5";
            (0,-15)*+{F(c,b)}="6";
            (0,-30)*+{F(c,c)}="7";
            (30,-45)*+{G(c,c)}="8";
            (30,-15)*+{F(b,b)}="a";
            (60,-30)*+{G(b,b)}="b";
            (30,-30)*+{G(c,b)}="c";
            (60,-15)*+{G(b,a)}="d";
            {\ar^{F(g,1)} "1" ; "2"};
            {\ar^{F(f,1)} "2" ; "3"};
            {\ar^{\beta_{aa}} "3" ; "4"};
            {\ar^{\gamma_a} "4" ; "5"};
            {\ar_{F(1,f)} "1" ; "6"};
            {\ar_{F(1,g)} "6" ; "7"};
            {\ar_{\beta_{cc}} "7" ; "8"};
            {\ar_{\gamma_{c}} "8" ; "5"};
            {\ar|{F(1,f)} "2" ; "a"};
            {\ar|{F(g,1)} "6" ; "a"};
            {\ar|{\beta_{bb}} "a" ; "b"};
            {\ar|{\gamma_b} "b" ; "5"};
            {\ar|{\beta_{cb}} "6" ; "c"};
            {\ar|{G(g,1)} "c" ; "b"};
            {\ar|{G(1,g)} "c" ; "8"};
            {\ar|{\beta_{ba}} "2" ; "d"};
            {\ar|{G(f,1)} "d" ; "4"};
            {\ar|{G(1,f)} "d" ; "b"};
            (15,-7.5)*+{\cong};
            {\ar@{=>}^{\beta_{f1}} (60,-5.5) ; (60,-9.5)};
            {\ar@{=>}^{\beta_{1f}^{-1}} (45,-13) ; (45,-17)};
            {\ar@{=>}^{\beta_{g1}} (30,-20.5) ; (30,-24.5)};
            {\ar@{=>}^{\beta_{1g}^{-1}} (15,-28) ; (15,-32)};
            {\ar@{=>}^{\gamma_{f}} (75,-25) ; (75,-29)};
            {\ar@{=>}^{\gamma_g} (48,-35.5) ; (48,-39.5)};
        \endxy
    \]
The right-hand diagram of EP1 is given by the following diagram.
    \[
        \xy
            (0,0)*+{F(c,a)}="1";
            (60,0)*+{F(b,a)}="2";
            (90,-15)*+{F(a,a)}="3";
            (90,-30)*+{G(a,a)}="4";
            (90,-45)*+{H}="5";
            (0,-30)*+{F(c,b)}="6";
            (30,-45)*+{F(c,c)}="7";
            (60,-45)*+{G(c,c)}="8";
            (60,-30)*+{G(c,a)}="a";
            {\ar^{F(g,1)} "1" ; "2"};
            {\ar^{F(f,1)} "2" ; "3"};
            {\ar^{\beta_{aa}} "3" ; "4"};
            {\ar^{\gamma_a} "4" ; "5"};
            {\ar_{F(1,f)} "1" ; "6"};
            {\ar_{F(1,g)} "6" ; "7"};
            {\ar_{\beta_{cc}} "7" ; "8"};
            {\ar_{\gamma_{c}} "8" ; "5"};
            {\ar|{\beta_{ca}} "1" ; "a"};
            {\ar|{G(gf,1)} "a" ; "4"};
            {\ar|{G(1,gf)} "a" ; "8"};
            {\ar|{F(gf,1)} "1" ; "3"};
            {\ar|{F(1,gf)} "1" ; "7"};
            (60,-5)*+{\cong};
            (10,-30)*+{\cong};
            {\ar@{=>}^{\beta_{gf,1}} (60,-18) ; (60,-22)};
            {\ar@{=>}^{\beta_{1,gf}^{-1}} (30,-28) ; (30,-32)};
            {\ar@{=>}^{\gamma_{gf}} (75,-35.5) ; (75,-39.5)};
        \endxy
    \]

The coherence cells in the top left of the first diagram allow us to use the composition axiom PS1 for $\beta$ to replace $\beta_{g1} \ast 1_{F(1,f)}$ and $1_{F(g,1)} \ast \beta_{1f}^{-1}$ with the corresponding components of $\beta$ and coherence cells for $G$ on the composites $G(g,1) \cdot G(1,f)$ and $G(1,f) \cdot G(g,1)$. The axiom PS2 for $\beta$, specifically on the unitors in $\m{A}^{op} \times \m{A}$, gives a new diagram where $\beta_{g,f}$ meets its inverse, leaving appropriate coherence cells in the square adjoining $\gamma_g$ and $\gamma_f$ to apply $EP1$ for $\gamma$. This leaves a diagram with $\gamma_{gf}$ in the lower right corner, at which point we use instances of the composition axiom for $\beta$ followed by its naturality on unitors again to yield the final diagram. The rest of the axioms are simple to check.
\end{proof}

The following lemma is the opposite of the previous lemma, having an extrapseudonatural transformation out of a constant pseudofunctor. Graphically the lemma corresponds to the following `stalagmite' diagram.
\begin{center}
\begin{tikzpicture}[y=0.80pt, x=0.80pt, yscale=2.500000, xscale=2.500000, inner sep=0pt, outer sep=0pt]
\path[draw=black,line join=miter,line cap=butt,even odd rule,line width=0.800pt]
  (440.0000,518) .. controls (440.0000,518) and (440.0378,527.4384) ..
  (449.8967,527.4384) .. controls (459.9800,527.4384) and (459.9056,518) ..
  (459.9056,518);
  \path[draw=black,line join=miter, line cap=butt, even odd rule, line width=0.800pt] (440.0000,512) -- (440.0000,484);
  \path[draw=black,line join=miter, line cap=butt, even odd rule, line width=0.800pt] (459.9056,512) -- (459.9056,484);
  \path[draw=black,line join=miter, line cap=butt, even odd rule, line width=0.800pt] (465,500) -- (469,500);
    \path[draw=black,line join=miter, line cap=butt, even odd rule, line width=0.800pt] (465,501) -- (469,501);
    \path[draw=black,line join=miter,line cap=butt,even odd rule,line width=0.800pt]
  (472.0000,496) .. controls (472.0000,496) and (472.0378,505.4384) ..
  (481.8967,505.4384) .. controls (491.9800,505.4384) and (491.9056,496) ..
  (491.9056,496);
%
\node() at (442,515){$\m{B}^{op}$};
\node() at (442,480.7){$\m{B}^{op}$};
\node() at (474,493){$\m{B}^{op}$};
\node() at (460,515){$\m{B}$};
\node() at (492,493){$\m{B}$};
\node() at (460,480.7){$\m{B}$};

\end{tikzpicture}
\end{center}
\begin{lemma}
Let $F \in \m{A}$ and let $G$, $H \colon \m{B}^{op} \times \m{B} \rightarrow \m{C}$ be pseuodfunctors. Suppose that $\beta \colon F \carrow G$ is an extrapseudonatural transformation and that $\gamma \colon G \Rightarrow H$ is a pseudonatural transformation. Then there is an extrapseudonatural transformation from $F$ to $H$ given by composites of the cells constituting $\beta$ and $\gamma$.
\end{lemma}
\begin{proof}
This proof is analogous to that of the previous lemma, only this time the axiom EP2 is the involved part.
\end{proof}

The final composition lemma corresponds to a graphical `yanking' of strings.
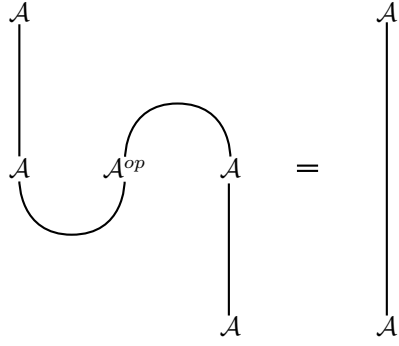
\begin{center}

\begin{tikzpicture}[y=0.80pt, x=0.80pt, yscale=-2.500000, xscale=2.500000, inner sep=0pt, outer sep=0pt]
\path[draw=black,line join=miter,line cap=butt,even odd rule,line width=0.800pt]
  (425.5000,607.1122) .. controls (425.5000,607.1122) and (425.5378,617.1884) ..
  (435.3967,617.1884) .. controls (445.4800,617.1884) and (445.4056,607.1564) ..
  (445.4056,607.1564);
\path[draw=black,line join=miter,line cap=butt,even odd rule,line width=0.800pt]
  (445.5000,602.4365) .. controls (445.5000,602.4365) and (445.5378,592.3603) ..
  (455.3967,592.3603) .. controls (465.4800,592.3603) and (465.4056,602.3923) ..
  (465.4056,602.3923);
\path[draw=black,line join=miter,line cap=butt,even odd rule,line width=0.800pt]
  (425.5000,577.3622) -- (425.5000,602.3622);
\path[draw=black,line join=miter,line cap=butt,even odd rule,line width=0.800pt]
  (465.1250,607.4872) -- (465.1250,632.4872);
\path[draw=black,line join=miter,line cap=butt,even odd rule,line width=0.800pt]
  (495.0000,576.9872) -- (495.0000,632.4872);
\node() at (425.5000,604.6122){$\m{A}$};
\node() at (445.4056,604.6122){$\m{A}^{op}$};
\node() at (465.4056,604.6122){$\m{A}$};
\node() at (425.5000, 575){$\m{A}$};
\node() at (495, 575){$\m{A}$};
\node() at (495, 634.5){$\m{A}$};
\node() at (465.4056, 634.5){$\m{A}$};
  \path[draw=black,line join=miter, line cap=butt, even odd rule, line width=0.800pt] (478,604.1122) -- (482,604.1122);
    \path[draw=black,line join=miter, line cap=butt, even odd rule, line width=0.800pt] (478,605.1122) -- (482,605.1122);
\end{tikzpicture}
\end{center}

\begin{lemma}
Let $F,H \colon \m{A} \rightarrow \m{B}$ and $G \colon \m{A} \times \m{A}^{op} \times \m{A} \rightarrow \m{B}$ be pseudofunctors. Suppose that for each $a \in \m{A}$
    \[
        \beta_{b} \colon Fb \carrow G(-,-,b), \gamma_{a-} \colon G(a,-,-) \carrow Ha
    \]
are extrapseudonatural transformations and that
    \[
        \beta_{a-} \colon F \Rightarrow G(a,a,-), \gamma_{-b} \colon G(-,b,b) \Rightarrow H
    \]
are pseudonatural transformations such that $(\beta_{a-})_b = (\beta_{b})_a$ and $(\gamma_{b})_a = (\gamma_{-a})_b$ for all $a, b \in \m{A}$. Then there is a pseudonatural transformation from $F$ to $H$ given by composites of the cells constituting $\beta$ and $\gamma$.
\end{lemma}
\begin{proof}
We need to construct a pseudonatural transformation from $F$ to $H$. The $1$-cell components are given by $\delta_{a} = \gamma_{aa} \cdot \beta_{aa}$, whilst the $2$-cell component, $\delta_f$, for some $f \colon a \rightarrow b$ in $\m{A}$ is given by the pasting diagram below.
    
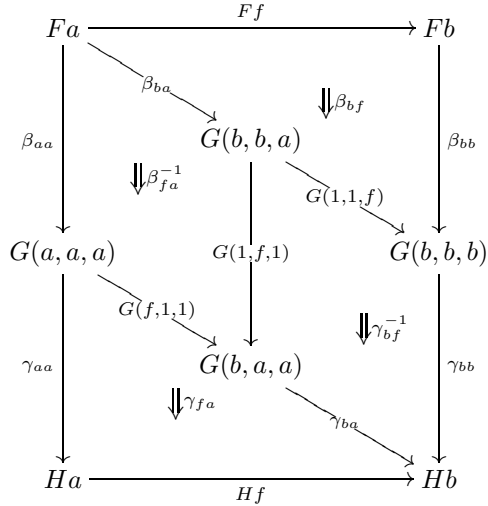
\[
        \xy
            (0,0)*+{Fa}="1";
            (50,0)*+{Fb}="2";
            (50,-30)*+{G(b,b,b)}="3";
            (50,-60)*+{Hb}="4";
            (0,-30)*+{G(a,a,a)}="5";
            (0,-60)*+{Ha}="6";
            (25,-15)*+{G(b,b,a)}="7";
            (25,-45)*+{G(b,a,a)}="8";
            {\ar^{Ff} "1" ; "2"};
            {\ar^{\beta_{bb}} "2" ; "3"};
            {\ar^{\gamma_{bb}} "3" ; "4"};
            {\ar_{\beta_{aa}} "1" ; "5"};
            {\ar_{\gamma_{aa}} "5" ; "6"};
            {\ar_{Hf} "6" ; "4"};
            {\ar|{\beta_{ba}} "1" ; "7"};
            {\ar|{G(1,1,f)} "7" ; "3"};
            {\ar|{G(1,f,1)} "7" ; "8"};
            {\ar|{G(f,1,1)} "5" ; "8"};
            {\ar|{\gamma_{ba}} "8" ; "4"};
            {\ar@{=>}^{\beta_{bf}} (35,-8) ; (35,-12)};
            {\ar@{=>}^{\gamma_{fa}} (15,-48) ; (15,-52)};
            {\ar@{=>}^{\beta_{fa}^{-1}} (10,-18) ; (10,-22)};
            {\ar@{=>}^{\gamma_{bf}^{-1}} (40,-38) ; (40,-42)};
        \endxy
    \]
Similar to the proof of the previous lemmas, the proof of the composition axiom PS1 relies on a sequence of pasting diagrams. The first diagram in the sequence is given by the component $\delta_{gf}$ with the coherence cells for $F$ and $G$ applied on the top and bottom.
    
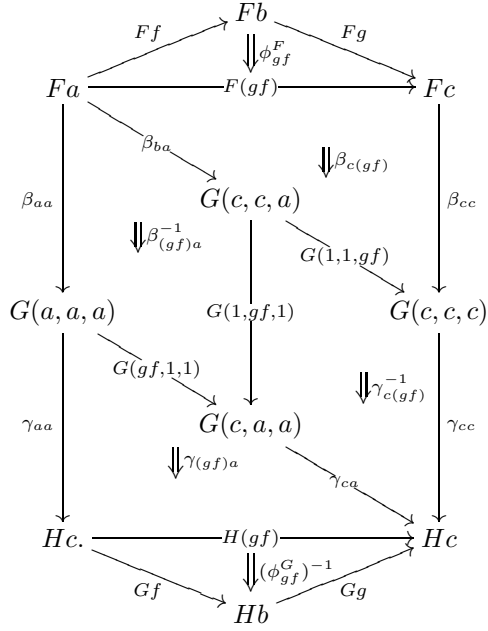
\[
        \xy
            (0,0)*+{Fa}="1";
            (50,0)*+{Fc}="2";
            (50,-30)*+{G(c,c,c)}="3";
            (50,-60)*+{Hc}="4";
            (0,-30)*+{G(a,a,a)}="5";
            (0,-60)*+{Hc.}="6";
            (25,-15)*+{G(c,c,a)}="7";
            (25,-45)*+{G(c,a,a)}="8";
            (25,10)*+{Fb}="9";
            (25,-70)*+{Hb}="10";
            {\ar|{F(gf)} "1" ; "2"};
            {\ar^{\beta_{cc}} "2" ; "3"};
            {\ar^{\gamma_{cc}} "3" ; "4"};
            {\ar_{\beta_{aa}} "1" ; "5"};
            {\ar_{\gamma_{aa}} "5" ; "6"};
            {\ar|{H(gf)} "6" ; "4"};
            {\ar|{\beta_{ba}} "1" ; "7"};
            {\ar|{G(1,1,gf)} "7" ; "3"};
            {\ar|{G(1,gf,1)} "7" ; "8"};
            {\ar|{G(gf,1,1)} "5" ; "8"};
            {\ar|{\gamma_{ca}} "8" ; "4"};
            {\ar^{Ff} "1" ; "9"};
            {\ar^{Fg} "9" ; "2"};
            {\ar_{Gf} "6" ; "10"};
            {\ar_{Gg} "10" ; "4"};
            {\ar@{=>}^{\beta_{c(gf)}} (35,-8) ; (35,-12)};
            {\ar@{=>}^{\gamma_{(gf)a}} (15,-48) ; (15,-52)};
            {\ar@{=>}^{\beta_{(gf)a}^{-1}} (10,-18) ; (10,-22)};
            {\ar@{=>}^{\gamma_{c(gf)}^{-1}} (40,-38) ; (40,-42)};
            {\ar@{=>}^{\phi^F_{gf}} (25,7) ; (25,2)};
            {\ar@{=>}^{(\phi^G_{gf})^{-1}} (25,-62) ; (25,-67)};
        \endxy
    \]
The final diagram is given below.
    
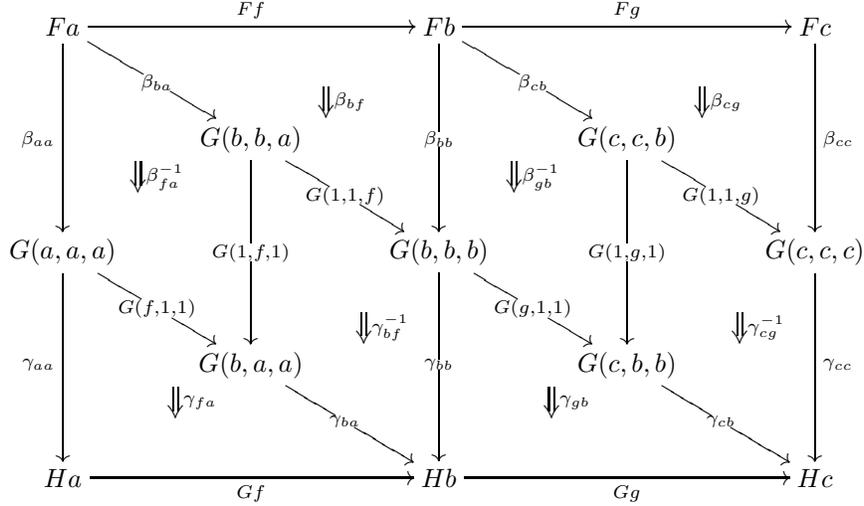
\[
        \xy
            (0,0)*+{Fa}="1";
            (100,0)*+{Fc}="2";
            (100,-30)*+{G(c,c,c)}="3";
            (100,-60)*+{Hc}="4";
            (0,-30)*+{G(a,a,a)}="5";
            (0,-60)*+{Ha}="6";
            (50,-30)*+{G(b,b,b)}="7";
            (25,-15)*+{G(b,b,a)}="8";
            (50,0)*+{Fb}="9";
            (50,-60)*+{Hb}="10";
            (25,-45)*+{G(b,a,a)}="11";
            (75,-15)*+{G(c,c,b)}="12";
            (75,-45)*+{G(c,b,b)}="13";
            {\ar^{\beta_{cc}} "2" ; "3"};
            {\ar^{\gamma_{cc}} "3" ; "4"};
            {\ar_{\beta_{aa}} "1" ; "5"};
            {\ar_{\gamma_{aa}} "5" ; "6"};
            {\ar^{Ff} "1" ; "9"};
            {\ar^{Fg} "9" ; "2"};
            {\ar_{Gf} "6" ; "10"};
            {\ar_{Gg} "10" ; "4"};
            {\ar|{\beta_{ba}} "1" ; "8"};
            {\ar|{G(1,1,f)} "8" ; "7"};
            {\ar|{G(1,f,1)} "8" ; "11"};
            {\ar|{G(f,1,1)} "5" ; "11"};
            {\ar|{\gamma_{ba}} "11" ; "10"};
            {\ar|{\beta_{bb}} "9" ; "7"};
            {\ar|{\gamma_{bb}} "7" ; "10"};
            {\ar|{\beta_{cb}} "9" ; "12"};
            {\ar|{G(1,1,g)} "12" ; "3"};
            {\ar|{G(1,g,1)} "12" ; "13"};
            {\ar|{G(g,1,1)} "7" ; "13"};
            {\ar|{\gamma_{cb}} "13" ; "4"};
            {\ar@{=>}^{\beta_{bf}} (35,-8) ; (35,-12)};
            {\ar@{=>}^{\gamma_{fa}} (15,-48) ; (15,-52)};
            {\ar@{=>}^{\beta_{fa}^{-1}} (10,-18) ; (10,-22)};
            {\ar@{=>}^{\gamma_{bf}^{-1}} (40,-38) ; (40,-42)};
            {\ar@{=>}^{\beta_{cg}} (85,-8) ; (85,-12)};
            {\ar@{=>}^{\gamma_{gb}} (65,-48) ; (65,-52)};
            {\ar@{=>}^{\beta_{gb}^{-1}} (60,-18) ; (60,-22)};
            {\ar@{=>}^{\gamma_{cg}^{-1}} (90,-38) ; (90,-42)};
        \endxy
    \]
The first step is to use the composition axioms (PS1) for the pseudonatural components of $\beta$ and $\gamma$, replacing the components at $gf$ with those for $g$ and $f$, whilst also introducing the composition coherence cells for $G(-,a,a)$ and $G(c,c,-)$. Now we can apply axioms EP2 and EP1, respectively, for the extrapseudonatural components of $\beta$ and $\gamma$. This gives a diagram with extrapseudonatural components for $\beta$ and $\gamma$ at $g$ and $f$, along with a large number of coherence cells in the middle of the diagram. At this point many of the coherence cells cancel out and we are then able to apply axioms EP4 and EP3, respectively, for the mixed components of $\beta$ and $\gamma$. The coherence cells introduced by the use of these axioms then cancel in the middle of the diagram, yielding the second diagram pictured above.

Proving the unit axiom, PS3, is fairly straightforward. Writing out the component $\delta_{id_a}$ we find that a large amount of the diagram is made of identities from the extrapseudonaturality axiom EP5. The remainder of the proof relies on the unit axioms for the bicategory $\m{B}$. The simplest part of the proof is to show that PS2 holds. A simple chase of $2$-cells through a diagram ensures that this holds.
\end{proof}

\section{Bicodescent Objects}\label{bicodsection}
Descent objects, the dual notion to codescent objects, first appeared in \cite{Str76} before being formally defined by Street in \cite{Str87}. We will base our definition of bicodescent object upon that given in \cite{Lac02}, where codescent objects are used to study coherence for the algebras of $2$-monads. Each of these treatments of codescent objects goes on to define them as weighted colimits, whereas our weaker notion of bicodescent object, being a bicolimit, has no description using weights. However, one could investigate the connection between weighted bicolimits and bicoends, following \cite{Str80}, \cite{Str87}. We will go on to recast the definition of bicodescent object as a universal object amongst extrapseudonatural transformations allowing us to obtain a Fubini theorem for bicodescent objects.

\begin{definition}\label{def:bicod1}
Let $\m{B}$ be a bicategory. Coherence data consists of a diagram
    \[
            \xymatrix{
            X_1  \ar[r]|>>>>>{v} & X_2 \ar@<1.ex>[l]^>>>>>{w} \ar@<-1.ex>[l]_>>>>>{u}& \ar@<1.ex>[l]^>>>>>>{r}\ar[l]|>>>>>>{q} \ar@<-1.ex>[l]_>>>>>>{p}X_3.
            }
    \]
in $\m{B}$ along with invertible $2$-cells
    \begin{align*}
        \delta \colon uv \Rightarrow &id_{X_1}, \gamma \colon id_{X_1} \Rightarrow wv, \kappa \colon up \Rightarrow uq,\\
        &\lambda \colon wr \Rightarrow wq, \rho \colon ur \Rightarrow wp.
    \end{align*}
The \emph{bicodescent object} of this coherence data consists of a $0$-cell $X$, a $1$-cell $x \colon X_1 \rightarrow X$, and an invertible $2$-cell $\chi \colon xu \Rightarrow xw$ in $\m{B}$ satisfying the following axioms.

\begin{itemize}
    \item[BC1] The following pasting diagrams are equal.
        \[
            \xy
                (0,0)*+{X_3}="a";
                (45,0)*+{X_2}="b";
                (-15,-12)*+{X_2}="c";
                (15,-12)*+{X_2}="d";
                (30,-12)*+{X_1}="e";
                (60,-12)*+{X_1}="f";
                (0,-24)*+{X_1}="g";
                (45,-24)*+{X;}="h";
                {\ar^{p} "a" ; "b"};
                {\ar^{u} "b" ; "f"};
                {\ar^{x} "f" ; "h"};
                {\ar_{q} "a" ; "c"};
                {\ar|{r} "a" ; "d"};
                {\ar|{u} "d" ; "e"};
                {\ar|{w} "b" ; "e"};
                {\ar_{w} "c" ; "g"};
                {\ar_{x} "g" ; "h"};
                {\ar|{x} "e" ; "h"};
                {\ar|{w} "d" ; "g"};
                {\ar@{=>}^{\rho} (22.5,-4) ; (22.5,-8)};
                {\ar@{=>}^{\chi} (22.5,-16) ; (22.6, -20)};
                {\ar@{=>}^{\lambda} (0,-10) ; (0,-14)};
                {\ar@{=>}^{\chi} (45,-10) ; (45,-14)};
                (0,-40)*+{X_3}="1";
                (45,-40)*+{X_2}="2";
                (-15,-52)*+{X_2}="3";
                (60,-52)*+{X_1}="4";
                (0,-64)*+{X_1}="5";
                (45,-64)*+{X}="6";
                {\ar^{p} "1" ; "2"};
                {\ar^{u} "2" ; "4"};
                {\ar^{x} "4" ; "6"};
                {\ar_{q} "1" ; "3"};
                {\ar_{w} "3" ; "5"};
                {\ar_{x} "5" ; "6"};
                {\ar|{u} "3" ; "4"};
                {\ar@{=>}^{\kappa} (22.5,-44) ; (22.5,-48)};
                {\ar@{=>}^{\chi} (22.5,-56) ; (22.5,-60)};
                {\ar@{=} (22.5,-30.5) ; (22.5,-33.5)};
            \endxy
        \]
        \item[BC2] The following pasting diagrams are equal.
        \[
            \xy
                (0,0)*+{X_1}="a";
                (15,0)*+{X_2}="b";
                (30,12)*+{X_1}="c";
                (30,-12)*+{X_1}="d";
                (45,0)*+{X}="e";
                {\ar^{v} "a" ; "b"};
                {\ar^{u} "b" ; "c"};
                {\ar_{w} "b" ; "d"};
                {\ar^{x} "c" ; "e"};
                {\ar_{x} "d" ; "e"};
                {\ar@{=>}^{\chi} (30,2) ; (30,-2)};
                (60,0)*+{X_1}="1";
                (75,12)*+{X_2}="2";
                (75,-12)*+{X_2}="3";
                (90,0)*+{X_1}="4";
                (105,0)*+{X}="5";
                {\ar^{v} "1" ; "2"};
                {\ar_{v} "1" ; "3"};
                {\ar^{u} "2"; "4"};
                {\ar_{w} "3" ; "4"};
                {\ar^{x} "4" ; "5"};
                {\ar|{id_{X_1}} "1" ; "4"};
                {\ar@{=>}^{\delta} (75,8) ; (75,4)};
                {\ar@{=>}^{\gamma} (75,-4) ; (75,-8)};
                {\ar@{=} (50.5,0) ; (54.5,0)};
            \endxy
        \]
        \item[BC3] Given any other $0$-cell $Y$, $1$-cell $y \colon X_1 \rightarrow Y$, and $2$-cell $\upsilon \colon yu \Rightarrow yw$ which satisfy the previous two axioms, there exists a $1$-cell $h \colon X \rightarrow Y$ and an isomorphism $\zeta \colon hx \Rightarrow y$ such that the following pasting diagrams are equal.
                \[
                    \xy
                        (0,0)*{X_2}="a";
                        (15,0)*+{X_1}="b";
                        (30,0)*+{X}="c";
                        (7.5,-12)*+{X_1}="d";
                        (22.5,-12)*{Y}="e";
                        {\ar^{u} "a" ; "b"};
                        {\ar^{x} "b" ; "c"};
                        {\ar^{h} "c" ; "e"};
                        {\ar_{y} "d" ; "e"};
                        {\ar|{y} "b" ; "e"};
                        {\ar_{w} "a" ; "d"};
                        {\ar@{=>}^{\zeta} (22.5,-3) ; (22.5,-7)};
                        {\ar@{=>}^{\upsilon} (10,-4) ; (12.5,-8)};
                        (50,-6)*{X_1}="a";
                        (50,6)*+{X_2}="b";
                        (65,6)*+{X_1}="c";
                        (57.5,-18)*+{Y}="d";
                        (65,-6)*{X}="e";
                        {\ar^{u} "b" ; "c"};
                        {\ar_{w} "b" ; "a"};
                        {\ar^{x} "c" ; "e"};
                        {\ar|{x} "a" ; "e"};
                        {\ar^{h} "e" ; "d"};
                        {\ar_{y} "a" ; "d"};
                        {\ar@{=>}^{\chi} (57.5,2) ; (57.5,-2)};
                        {\ar@{=>}^{\zeta} (57.5,-8) ; (57.5,-12)};
                        {\ar@{=} (38,-6) ; (42,-6)};
                    \endxy
                \]
        \item[BC4] Given a $0$-cell $Y$, $1$-cells $h, k \colon X \Rightarrow Y$, and a $2$-cell $\beta \colon h \Rightarrow kx$ satisfying
                \[
                    \xy
                    (0,0)*+{X_2}="a";
                    (12,12)*+{X_1}="b";
                    (24,24)*+{X}="c";
                    (24,0)*+{X}="d";
                    (36,12)*+{Y}="e";
                    (12,-12)*+{X_1}="f";
                    {\ar^{u} "a" ; "b"};
                    {\ar^{x} "b" ; "c"};
                    {\ar^{h} "c" ; "e"};
                    {\ar_{w} "a" ; "f"};
                    {\ar_{x} "f" ; "d"};
                    {\ar_{k} "d" ; "e"};
                    {\ar^{x} "b" ; "d"};
                    {\ar@{=>}^{\beta} (24,14) ; (24,10)};
                    {\ar@{=>}^{\chi} (12,2) ; (12,-2)};
                    (50,12)*+{X_2}="1";
                    (62,24)*+{X_1}="2";
                    (74,12)*+{X}="3";
                    (74,-12)*+{X}="4";
                    (86,0)*+{Y}="5";
                    (62,0)*+{X_1}="6";
                    {\ar^{u} "1" ; "2"};
                    {\ar^{x} "2" ; "3"};
                    {\ar^{h} "3" ; "5"};
                    {\ar_{w} "1" ; "6"};
                    {\ar_{x} "6" ; "4"};
                    {\ar_{k} "4" ; "5"};
                    {\ar^{x} "6" ; "3"};
                    {\ar@{=>}^{\chi} (62,14) ; (62,10)};
                    {\ar@{=>}^{\beta} (74,2) ; (74,-2)};
                    {\ar@{=} (41,0) ; (45,0)};
                    \endxy
                \]
            there exists a unique $2$-cell $\beta' \colon h \Rightarrow k$ such that $\beta' \ast 1_{x} = \beta$.
\end{itemize}
\end{definition}

Now we have defined bicodescent objects we will liken them to coends. Coends can be described as a colimit for functors of the form $F \colon \m{A}^{op} \times \m{A} \rightarrow \m{C}$, being given as the coequalizer
    \[
            \xymatrix@1{
            \int^{a} F(a,a) \, & {\displaystyle \coprod_{a} F(a,a)} \ar^{i}[l] & \ar@<1.ex>[l]^>>>>>{\rho} \ar@<-1.ex>[l]_>>>>>{\lambda} {\displaystyle \coprod_{\overset{f}{a \rightarrow}b} F(b,a)}
            }
    \]
where $\lambda$ and $\rho$ act in a similar manner to $u$ and $w$ below. In our case, a bicodescent object will be a bicolimit for pseudofunctors of the form $P \colon \m{B}^{op} \times \m{B} \rightarrow \m{C}$. The previous definition only has two axioms but requires setting up a lot of data, whereas using extrapseudonatural transformations requires little in the specification of data with the trade-off of checking a few more axioms.

Given a pseudofunctor $P \colon \m{B}^{op} \times \m{B} \rightarrow \m{C}$ we describe its coherence data as follows. 
    \[
            \xymatrix@1{
            {\displaystyle \coprod_{a \in ob\m{B}} P(a,a) } \ar@<1.ex>@{<-}[r]^>>>>>{u} \ar@<-1.ex>@{<-}[r]_>>>>>{w} \,\,& \,\,\ar@{<-}[l]|>>>>>{v} {\displaystyle \coprod_{f} P(b,a)} \ar@<1.ex>@{<-}[r]^>>>>>{p} \ar@{<-}[r]|>>>>>{q} \ar@<-1.ex>@{<-}[r]_>>>>>{r} \,\,&\,\, {\displaystyle \coprod_{\theta} P(c,a)}
            }
    \]
The middle coproduct is indexed over $1$-cells $f \colon a \rightarrow b$ while the last coproduct is indexed over $2$-cells $\theta \colon gf \rightarrow h$ for $1$-cells $f \colon a \rightarrow b$, $g \colon b \rightarrow c$, and $h \colon a \rightarrow c$.

In the following, the $1$-cells $I_a$ and $J_f$ are coproduct inclusions. The $1$-cell $u$ is determined by the $1$-cells
    \[
        P(b,a) \overset{P(f,1)}{\longrightarrow} P(a,a) \overset{I_{a}}{\longrightarrow} \coprod_a P(a,a),
    \]
the $1$-cell $w$ is determined by the $1$-cells
    \[
        P(b,a) \overset{P_{bf}}{\longrightarrow} P(b,b) \overset{I_b}{\longrightarrow} \coprod_a P(a,a),
    \]
and the $1$-cell $v$ is determined by the inclusion on identities. The $1$-cell $p$ is characterised by the $1$-cells
    \[
        P(c,a) \overset{P_{ga}}{\longrightarrow} P(b,a) \overset{J_f}{\longrightarrow} \coprod_{f \colon a \rightarrow b} P(b,a),
    \]
the $1$-cell $q$ is determined by the $1$-cells
    \[
        P(c,a) \overset{J_h}{\longrightarrow} \coprod_{f \colon a \rightarrow b} P(b,a),
    \]
and the $1$-cell $r$ is characterised by the $1$-cells
    \[
        P(c,a) \overset{P_{cf}}{\longrightarrow} P(c,b) \overset{J_g}{\longrightarrow} \coprod_{f \colon a \rightarrow b} P(b,a).
    \]

When we consider a pseudofunctor, say $F \colon \m{A} \rightarrow \m{B}$, we will write its coherence cells as follows. For composition, on $1$-cells $f \colon a \rightarrow b$, $g \colon b \rightarrow c$ in $\m{A}$, we write
    \[
        \phi^F_{gf} \colon F(g) \cdot F(f) \Rightarrow F(g\cdot f)
    \]
with
    \[
        \phi^F_{aa} = \phi^F_{id_a,id_a}
    \]
while for identities we write
    \[
        \phi^F_a \colon F(id_a) \Rightarrow id_{Fa}
    \]
where $a \in \m{A}$.
We now use the corresponding cells for $P$ in order to determine the coherence data for the bicodescent object of $P$. We will omit certain indices on the pseudofunctor coherence cells for $P$ when it is obvious. The $2$-cell $\delta \colon uv \Rightarrow id$ is determined by the $2$-cells
    \[
        \xy
            (0,0)*+{P(a,a)}="a";
            (25,0)*+{P(a,a)}="b";
            {\ar@/^1.25pc/^{P(id_a,id_a)} "a" ; "b"};
            {\ar@/_1.25pc/_{id_{P(a,a)}} "a" ; "b"};
            {\ar@{=>}^{\phi^P_{aa}} (12.5,2) ; (12.5,-2)};
        \endxy
    \]
Similarly the $2$-cell $\gamma \colon id \Rightarrow wv$ is characterised by the inverses of those that give $\delta$. The $2$-cell $\kappa \colon up \Rightarrow uq$ is characterised by the $2$-cells
    \[
        \xy
            (0,0)*+{P(c,a)}="a";
            (40,0)*+{P(a,a)}="d";
            (20,15)*+{P(b,a)}="b";
            %
            {\ar@/^0.5pc/^{P(g,1)} "a" ; "b"};
            {\ar@/^0.5pc/^{P(f,1)} "b" ; "d"};
            {\ar@/^1pc/|{P(gf, 11)} "a" ; "d"};
            {\ar@/_1pc/|{P(h, 1)} "a" ; "d"};
            {\ar@{=>}^{\phi^P} (20,11) ; (20,7)};
            {\ar@{=>}^{P(\theta,r_1)} (20,2) ; (20,-2)};
            %
        \endxy
    \]
and the $2$-cell $\lambda \colon wr \Rightarrow wq$ is characterised by the $2$-cells
    \[
        \xy
            (0,0)*+{P(c,a)}="a";
            (40,0)*+{P(c,c).}="d";
            (20,15)*+{P(c,b)}="b";
            %
            {\ar@/^0.5pc/^{P(1,f)} "a" ; "b"};
            {\ar@/^0.5pc/^{P(1,g)} "b" ; "d"};
            {\ar@/^1pc/|{P(11,gf)} "a" ; "d"};
            {\ar@/_1pc/|{P(1,h)} "a" ; "d"};
            {\ar@{=>}_{{(\phi^P)}} (20,11) ; (20,7)};
            {\ar@{=>}_{P(l_1,\theta)} (20,2) ; (20,-2)};
            %
        \endxy
    \]
The remaining $2$-cell, $\rho \colon ur \Rightarrow wp$, is characterised by the $2$-cells
    \[
        \xy
            (0,0)*+{P(c,a)}="a";
            (40,0)*+{P(b,b).}="d";
            (20,20)*+{P(c,b)}="b";
            (20,-20)*+{P(b,a)}="c";
            {\ar@/^1pc/^{P(1,f)} "a" ; "b"};
            {\ar@/^1pc/^{P(g,1)} "b" ; "d"};
            {\ar@/_1pc/_{P(g,1)} "a" ; "c"};
            {\ar@/_1pc/_{P(1,f)} "c" ; "d"};
            {\ar@/^2.2pc/|{P(1g,1f)} "a" ; "d"};
            {\ar@/_2.2pc/|{P(g1,f1)} "a" ; "d"};
            {\ar|{P(g,f)} "a" ; "d"};
            {\ar@{=>}^{\phi^P} (20,16) ; (20,12)};
            {\ar@{=>}^{P(r,l)} (20,6) ; (20,2)};
            {\ar@{=>}_{P(l, r)^{-1}} (20,-2) ; (20,-6)};
            {\ar@{=>}_{(\phi^P)^{-1}} (20,-12) ; (20,-16)};
        \endxy
    \]

The following results will characterise bicodescent objects as objects which are universal amongst extrapseudonatural transformations. By this we mean that the morphisms
    \[
        i_a \colon P(a,a) \rightarrow \CodP
    \]
are part of the data for an extrapseudonatural transformation, satisfying universal properties as in Definition \ref{def:epnatcod}.

\begin{lemma}\label{epnatcod}
Let $(Y,y,\upsilon)$ be data as in \ref{def:bicod1} satisfying only the axioms $BC1$ and $BC2$ for the coherence data of a pseudofunctor $P$. Then it is necessary and sufficient for the corresponding $(Y, y_a, \upsilon_f)$ to constitute an extrapseudonatural transformation.
\end{lemma}
\begin{proof}
Collectively the $(Y, y_a, \upsilon_f)$ give a triple $(Y, y, \upsilon)$. The axiom BC1 then follows from the axioms EP1, to change an $\upsilon_{gf}$ into a composite of $\upsilon_g$ and $\upsilon_f$, and EP6 for the modification-like property. For BC2 we see that one side of the pasting diagram corresponds to $\upsilon_{id_a} = id_{y_a \cdot P_{aa}}$ by EP5, whilst the other side is a composite of $\delta$ and $\gamma = \delta^{-1}$, giving the identity required.

Now we will show that the $1$-cells
    \[
        y_b \colon P(b,b) \rightarrow Y
    \]
along with $2$-cells, to be described, constitute an extrapseudonatural transformation. As $P$ is a pseudofunctor out of $\m{B}^{op} \times \m{B}$ then we require pseudonatural transformations between the pseudofunctor
    \begin{align*}
        \overline{\Delta}_{P(b,b)} \colon &\mathbf{1} \longrightarrow \m{C} \\
        &\cdot \longmapsto P(b,b) \\
        &1_{\cdot} \longmapsto P(1_b,1_b) \\
        &1_{1_{\cdot}} \longmapsto id_{P(1_b,1_b)}
    \end{align*}
and the constant pseudofunctor $\Delta_Y$. The $2$-cell components at the identity are given by
    \[
        \xy
            (0,0)*+{P(b,b)}="a";
            (25,0)*+{P(b,b)}="b";
            (0,-20)*+{Y}="c";
            (25,-20)*+{Y}="d";
            {\ar@/^1.5pc/^{P(1_b,1_b)} "a" ; "b"};
            {\ar_{id} "a" ; "b"};
            {\ar^{y_b} "b" ; "d"};
            {\ar_{y_b} "a" ; "c"};
            {\ar_{id} "c" ; "d"};
            {\ar|{y_b} "a" ; "d"};
            {\ar@{=>}^{\phi_{bb}} (12.5,5) ; (12.5,1)};
            {\ar@{=>}^{r_{y_b}} (18,-4) ; (14,-7)};
            {\ar@{=>}^{l_{y_b}^{-1}} (10,-11) ; (6,-14)}
        \endxy
    \]
which we will denote by $\overline{j}_{b}$.
Using the fact that in any bicategory the left and right unitors at the identity are equal, along with the naturality of many of the coherence cells we can see that this constitutes a pseudonatural transformation. The first collection of extrapseudonatural $2$-cells are given by
    \[
        \xy
            (0,0)*+{P(b,a)}="a";
            (25,0)*+{P(a,a)}="b";
            (0,-20)*+{P(b,b)}="c";
            (25,-20)*+{Y}="d";
            {\ar^{P(f,1)} "a" ; "b"};
            {\ar^{y_a} "b" ; "d"};
            {\ar_{P(1,f)} "a" ; "c"};
            {\ar_{y_b} "c" ; "d"};
            {\ar@{=>}^{\upsilon_f} (12.5,-8) ; (12.5,-12)};
        \endxy
    \]
where $\upsilon$ is the $2$-cell given in the triple. The second collection of extrapseudonatural $2$-cells is given by the identity on the $1$-cell $id_{Y} \cdot y_b$.

The axiom EP1 holds by an instance of BC1 using the identity $2$-cell $id \colon g \cdot f \Rightarrow gf$, whilst EP2 holds trivially. The third axiom, EP3, requires an equality of the following two pasting diagrams.
                    \[
                        \xy
                            (55,0)*+{P(b,a)}="a";
                            (70,-7)*+{P(b,a)}="b";
                            (70,-35)*+{P(a,a)}="c";
                            (55,-42)*+{Y}="d";
                            (40,-7)*+{P(b,b)}="e";
                            (40,-35)*+{Y}="f";
                            (55,-14)*+{P(b,b)}="g";
                            {\ar^{P(1,1)} "a" ; "b"};
                            {\ar^{P(f,1)} "b" ; "c"};
                            {\ar^{j_a} "c" ; "d"};
                            {\ar_{P(1,f)} "a" ; "e"};
                            {\ar_{j_b} "e" ; "f"};
                            {\ar_{id_Y} "f" ; "d"};
                            {\ar^{P(1,f)} "b" ; "g"};
                            {\ar|{j_b} "g" ; "d"};
                            {\ar_{P(1,1)} "e" ; "g"};
                            (55,-7)*+{\cong};
                            {\ar@{=>}^{j_f} (62.5,-22) ; (58.5,-25)};
                            {\ar@{=>}^{\overline{j}_{b}} (47.5,-22) ; (43.5,-25)};
                            (0,0)*+{P(b,a)}="1";
                            (15,-7)*+{P(b,a)}="2";
                            (15,-35)*+{P(a,a)}="3";
                            (0,-42)*+{Y}="4";
                            (-15,-7)*+{P(b,b)}="5";
                            (-15,-35)*+{Y}="6";
                            (0,-28)*+{P(a,a)}="7";
                            {\ar^{P(1,1)} "1" ; "2"};
                            {\ar^{P(f,1)} "2" ; "3"};
                            {\ar^{j_a} "3" ; "4"};
                            {\ar_{P(1,f)} "1" ; "5"};
                            {\ar_{j_b} "5" ; "6"};
                            {\ar_{id_Y} "6" ; "4"};
                            {\ar|{P(f,1)} "1" ; "7"};
                            {\ar^{P(1,1)} "7" ; "3"};
                            {\ar_{j_a} "7" ; "6"};
                            (7.5,-18.5)*+{\cong};
                            {\ar@{=>}^{j_f} (-7.5,-17) ; (-11.5,-20)};
                            {\ar@{=>}^{\overline{j_a}} (1,-33.5) ; (-3,-36.5)};
                            %
                        \endxy
                    \]
Written out as a commutative diagram of $2$-cells, including all coherence cells, this can plainly be seen to hold as a result of naturality of various coherence $2$-cells, unit axioms for $P$, and triangle identities in $\m{C}$. The fourth axiom, EP4, follows by a similar, though simpler, argument, whilst EP5 holds immediately due to $\delta$ and $\gamma$ being inverse to each other, giving
    \[
        \chi_{id_a} = id_{j_a \cdot P(1_a,1_a)}.
    \]
Clearly axiom EP7 holds since we are considering a constant pseudofunctor $\Delta_Y$. It remains then to check EP6 which requires, for each $\theta \colon g \Rightarrow g'$ between $g,g' \colon a \rightarrow b$ in $\m{B}$, an equality of pasting diagrams as follows.
 \[
                \xy
                    (0,0)*+{P(b,a)}="a";
                    (30,0)*+{P(a,a)}="b";
                    (0,-18)*+{P(b,b)}="c";
                    (30,-18)*+{Y}="d";
                    {\ar_{P(g',1)} "a" ; "b"};
                    {\ar_{P(1,g)} "a" ; "c"};
                    {\ar^{y_a} "b" ; "d"};
                    {\ar_{y_b} "c" ; "d"};
                    {\ar@/^2pc/^{P(g,1)} "a" ; "b"};
                    {\ar@{=>}^{\upsilon_{g'}} (14,-7) ; (14,-11)};
                    {\ar@{=>}^{P(\theta,1)} (12, 6) ; (12, 2)};
                    (70,0)*+{P(a,b)}="1";
                    (100,0)*+{P(a,a)}="2";
                    (70,-18)*+{P(b,b)}="3";
                    (100,-18)*+{Y}="4";
                    {\ar^{P(g,1)} "1" ; "2"};
                    {\ar@/^0.5pc/^{P(1,g)} "1" ; "3"};
                    {\ar^{y_{a}} "2" ; "4"};
                    {\ar_{y_{b}} "3" ; "4"};
                    {\ar@/_2pc/_{P(1,g')} "1" ; "3"};
                    {\ar@{=>}^{\upsilon_{g}} (87,-7) ; (87,-11)};
                    {\ar@{=>}_{P(1,\theta)} (68.5,-10) ; (64.5,-10)};
                    {\ar@{=} (45,-9) ; (48,-9)};
                \endxy
            \]
This is slightly tricky to prove but really relies on making a suitable choice of $2$-cell in $\m{B}$ when considering BC1. First we can check that EP6 holds for the $2$-cell $r_g \colon g \cdot id_a \Rightarrow g$, this relies on the fact that many of the pseudofunctor coherence cells for $P$, and the image of some unitors, can be cancelled in the resulting diagrams. If we then have a $2$-cell $\theta \colon g \Rightarrow g'$ then choosing the $2$-cell $\gamma \cdot r_g$
    \[
        \xy
            (0,0)*+{a}="a";
            (15,15)*+{a}="b";
            (30,0)*+{b}="c";
            {\ar@/^0.5pc/^{id_a} "a" ; "b"};
            {\ar@/^0.5pc/^{g} "b" ; "c"};
            {\ar@/^0.5pc/|{g} "a" ; "c"};
            {\ar@/_1.5pc/_{g'} "a" ; "c"};
            {\ar@{=>}^{r_g} (15,10) ; (15,6)};
            {\ar@{=>}^{\theta} (15,0) ; (15,-4)};
        \endxy
    \]
in an instance of BC1 proves that EP6 holds in all cases.
\end{proof}

We now describe the bicoend of $P$ as the universal extrapseudonatural transformation out of $P$, which we will later use as our definition of bicodescent object.

\begin{definition}\label{def:epnatcod}
Let $P \colon \m{B}^{op} \times \m{B} \rightarrow \m{C}$ be a pseudofunctor. The bicoend of $P$ is given by $i \colon P \carrow \int^{b} P(b,b)$ satisfying the following universal properties:
    \begin{itemize}
        \item EB1 Given another object $X$ with an extrapseudonatural transformation $j \colon P \carrow X$, there is a $1$-cell $\tilde{j} \colon \int^{b} P(b,b) \rightarrow X$ and isomorphisms $J_a \colon \tilde{j} \cdot i_a \cong j_a$ such that the following equality of pasting diagrams holds.
                \[
                    \xy
                        (0,0)*{P_{ba}}="a";
                        (20,0)*+{P_{aa}}="b";
                        (40,0)*+{\codP}="c";
                        (10,-12)*+{P_{bb}}="d";
                        (30,-12)*{Y}="e";
                        {\ar^{P_{f1}} "a" ; "b"};
                        {\ar^{i_a} "b" ; "c"};
                        {\ar^{\tilde{j}} "c" ; "e"};
                        {\ar_{j_b} "d" ; "e"};
                        {\ar|{j_a} "b" ; "e"};
                        {\ar_{P_{1f}} "a" ; "d"};
                        {\ar@{=>}^{J_a} (30,-3) ; (30,-7)};
                        {\ar@{=>}^{j_f} (12.5,-4) ; (15,-8)};
                        (60,-6)*{P_{bb}}="a";
                        (60,6)*+{P_{ba}}="b";
                        (85,6)*+{P_{aa}}="c";
                        (72.5,-18)*+{Y}="d";
                        (85,-6)*{\codP}="e";
                        {\ar^{P_{f1}} "b" ; "c"};
                        {\ar_{P_{1f}} "b" ; "a"};
                        {\ar^{i_a} "c" ; "e"};
                        {\ar|{i_b} "a" ; "e"};
                        {\ar^{\tilde{j}} "e" ; "d"};
                        {\ar_{j_b} "a" ; "d"};
                        {\ar@{=>}^{i_f} (72.5,2) ; (72.5,-2)};
                        {\ar@{=>}^{J_b} (72.5,-9) ; (72.5,-13)};
                        {\ar@{=} (48,-6) ; (52,-6)};
                    \endxy
                \]
        \item EB2 Given two $1$-cells $h$, $k \colon \codP \rightarrow Y$ and $2$-cells $\Gamma_a \colon h \cdot i_a \Rightarrow k \cdot i_a$ satisfying
                 \[
                    \xy
                    (0,0)*+{P_{ba}}="a";
                    (12,12)*+{P_{aa}}="b";
                    (24,24)*+{\codP}="c";
                    (24,0)*+{\codP}="d";
                    (36,12)*+{Y}="e";
                    (12,-12)*+{P_{bb}}="f";
                    {\ar^{P_{f1}} "a" ; "b"};
                    {\ar^{i_a} "b" ; "c"};
                    {\ar^{h} "c" ; "e"};
                    {\ar_{P_{1f}} "a" ; "f"};
                    {\ar_{i_b} "f" ; "d"};
                    {\ar_{k} "d" ; "e"};
                    {\ar|{i_a} "b" ; "d"};
                    {\ar@{=>}^{\Gamma_a} (24,14) ; (24,10)};
                    {\ar@{=>}^{i_f} (12,2) ; (12,-2)};
                    (50,12)*+{P_{ba}}="1";
                    (62,24)*+{P_{aa}}="2";
                    (74,12)*+{\codP}="3";
                    (74,-12)*+{\codP}="4";
                    (86,0)*+{Y}="5";
                    (62,0)*+{P_{bb}}="6";
                    {\ar^{P_{f1}} "1" ; "2"};
                    {\ar^{i_a} "2" ; "3"};
                    {\ar^{h} "3" ; "5"};
                    {\ar_{P_{1f}} "1" ; "6"};
                    {\ar_{i_b} "6" ; "4"};
                    {\ar_{k} "4" ; "5"};
                    {\ar|{i_b} "6" ; "3"};
                    {\ar@{=>}^{i_f} (62,14) ; (62,10)};
                    {\ar@{=>}^{\Gamma_b} (74,2) ; (74,-2)};
                    {\ar@{=} (41,0) ; (45,0)};
                    \endxy
                \]
 there is a unique $2$-cell $\gamma \colon h \Rightarrow k$ such that $\Gamma_a = \gamma \ast 1_{i_a}$ for all $a \in \m{A}$.
    \end{itemize}
 \end{definition}
\begin{lemma}
Let $P \colon \m{B}^{op} \times \m{B} \rightarrow \m{C}$ be a pseudofunctor and suppose that $i \colon P \carrow \int^{b} P(b,b)$ exists. Let $j \colon P \carrow X$ be another extrapseudonatural transformation which also satisfies the axioms EB1 and EB2. Then there is an adjoint equivalence between $\int^{b} P(b,b)$ and $X$.
\end{lemma}
\begin{proof}
Given two such objects as described it is simple to see that there are appropriate unit and counit isomorphisms $\tilde{i} \cdot \tilde{j} \cong 1_{\codP}$ and $\tilde{j} \cdot \tilde{i} \cong 1_X$ induced by axiom EB2 above by reliance on the conditions in axiom EB1. Since the induced $2$-cells are unique it is then a simple check to see that the triangle inequalities hold and there is an adjoint equivalence between $\CodP$ and $X$.
\end{proof}

In the sense of the above lemma we can consider bicodescent objects to be essentially unique.
\begin{proposition}
Let $P \colon \m{B}^{op} \times \m{B} \rightarrow \m{C}$ be a pseudofunctor. The bicodescent object corresponding to the coherence data for $P$ is equivalent to the bicoend of $P$.
\end{proposition}
\begin{proof}
By Lemma \ref{epnatcod} the triple of a bicodescent object for $P$, $(\CodP, x, \chi)$, provides an extrapseudonatural transformation $x \colon P \carrow \CodP$. It is clear that $\CodP$ satisfies the same universal properties as $\codP$. Similarly the extrapseudonatural transformation $i$ from $P$ to $\Delta_{\codP}$ satisfies all of the axioms of a bicodescent object for $P$, including the universal properties. We will describe how to show that these objects are equivalent.

The bicoend has invertible $2$-cells $i_f \colon i_a \cdot P_{fa} \Rightarrow i_b \cdot P_{bf}$, where $f \colon a \rightarrow b$ is a $1$-cell in $\m{B}$. Collectively this family of $2$-cells corresponds to a $2$-cell $i$.
    \[
        \xy
            (0,0)*+{\coprod_f P_{ba}}="a";
            (25,0)*+{\coprod_{a} P_{aa}}="b";
            (0,-20)*+{\coprod_{a} P_{aa}}="c";
            (25,-20)*+{\int^{b}P_{bb}}="d";
            {\ar^{u} "a" ; "b"};
            {\ar^{j} "b" ; "d"};
            {\ar_{w} "a" ; "c"};
            {\ar_{j} "c" ; "d"};
            {\ar@{=>}^{i} (12.5,-8) ; (12.5,-12)};
        \endxy
    \]
Each of the $1$-cells in the above diagram are induced using the universal properties of the displayed coproducts. The conditions in BC3 are met which induces a $2$-cell as below.
    \[
        \xy
            (0,0)*+{\CodP}="a";
            (25,0)*+{\coprod_{a} P_{aa}}="b";
            (0,-20)*+{\int^{a} P_{aa}}="c";
            {\ar_{j} "b" ; "a"};
            {\ar_{s} "a" ; "c"};
            {\ar^{j} "b" ; "c"};
            (10,-6)*+{\cong};
        \endxy
    \]

Similarly the bicodescent object $\CodP$ is an extrapseudonatural transformation, as previously described, satisfying the appropriate axioms. Since $\int^{a}P_{aa}$ is the universal such transformation there is an induced $1$-cell $t \colon \int^{a}P_{aa} \rightarrow \CodP$ satisfying the properties described in the previous definition. Our claim now is that $s$ and $t$ form an equivalence in $\m{B}$.

It is simple to check this claim. In analogous $1$-dimensional cases this would be immediate following from the uniqueness inherent in the $1$-dimensional universal property. Since the $1$-dimensional properties now no longer contain a uniqueness statement, we do not find that $s$ and $t$ are inverses but that we instead obtain isomorphisms $1 \cong st$ and $ts \cong 1$, as in the diagrams below.

    \[
        \xy
            (0,0)*+{\CodP}="a";
            (0,-20)*+{\int^{a} P_{aa}}="b";
            (0,-40)*+{\CodP}="c";
            (25,-20)*+{\coprod_{a \in \m{B}} P_{aa}}="d";
            {\ar_{x} "d" ; "a"};
            {\ar|{j} "d" ; "b"};
            {\ar^{x} "d" ; "c"};
            {\ar_{s} "a" ; "b"};
            {\ar_{t} "b" ; "c"};
            {\ar@/_3.5pc/_{id} "a" ; "c"};
            (9,-14)*+{\cong};
            (9,-26)*+{\cong};
            (-10,-20)*+{\cong};
        \endxy
    \]
    \[
        \xy
            (0,-20)*+{\CodP}="b";
            (0,0)*+{\int^{a} P_{aa}}="a";
            (0,-40)*+{\int^{a} P_{aa}}="c";
            (25,-20)*+{\coprod_{a \in \m{B}} P_{aa}}="d";
            {\ar_{i_a} "d" ; "a"};
            {\ar|{x_a} "d" ; "b"};
            {\ar^{i_a} "d" ; "c"};
            {\ar_{t} "a" ; "b"};
            {\ar_{s} "b" ; "c"};
            {\ar@/_3.5pc/_{id} "a" ; "c"};
            (9,-14)*+{\cong};
            (9,-26)*+{\cong};
            (-10,-20)*+{\cong};
        \endxy
    \]
These isomorphisms can then be used to show that the two objects $\CodP$ and $\int^{a}P_{aa}$ are equivalent in $\m{B}$.
\end{proof}
We will use the notation $i \colon P \carrow \int^{b} P(b,b)$ to refer to the bicodescent object corresponding to a pseudofunctor $P$.

\begin{lemma}\label{psfunctor}
 Let $P \colon \m{A} \times \m{B}^{op} \times  \m{B} \rightarrow \m{C}$ be a pseudofunctor. Assume that, for each $a \in \m{A}$, the bicodescent object
    \[
         j^{a} \colon P_{a--} \carrow \int^{b} P(a,b,b)
    \]
 exists in $\m{C}$. Then
    \[
        a \longmapsto \int^{b} P(a,b,b)
    \]
 is the object part of a pseudofunctor
    \[
        \int^{b} P(-,b,b) \colon \m{A} \rightarrow \m{C}.
    \]
 \end{lemma}
 \begin{proof}
Each $f \colon a \rightarrow a'$ in $\m{A}$ gives a pseudonatural transformation
    \[
        P_{f--} \colon P_{a--} \Rightarrow P_{a'--}.
    \]
By Lemma \ref{comp1} there is then an extrapseudonatural transformation
    \[
        j^{a'} \cdot P_{f--} \colon P_{a--} \carrow \int^{b} P_{abb},
    \]
inducing the following invertible $2$-cells.
    \[
        \xy
            (0,0)*+{P_{abb}}="a";
            (25,0)*+{P_{a'bb}}="b";
            (0,-20)*+{\int^{b} P_{abb}}="c";
            (25,-20)*+{\int^{b}P_{a'bb}}="d";
            {\ar^{P_{fbb}} "a" ; "b"};
            {\ar^{j^{a'}_b} "b" ; "d"};
            {\ar_{j^a_b} "a" ; "c"};
            {\ar_{\int^{b} P_{fbb}} "c" ; "d"};
            {\ar@{=>}^{j^f_b} (12.5,-8) ; (12.5,-12)};
        \endxy
    \]
The coherence cells of $P$ along with these $j^f_b$ induce coherence cells for these new $1$-cells. This can be seen more clearly in the remarks following the proof. The uniqueness in the $2$-dimensional universal property of each $\int^{b} P_{abb}$ shows that each of the axioms for a pseudofunctor are satisfied. Furthermore, the above $2$-cells constitute pseudonatural transformations $j_b \colon P_{-bb} \Rightarrow \int^{b} P_{-bb}$.
 \end{proof}

 For reference, we will describe the properties of the coherence cells for the pseudofunctors $\int^{b} P_{-bb} \colon \m{A} \rightarrow \m{C}$. The inverse of the invertible $2$-cell
    \[
        \xy
            (0,0)*+{\int^{b} P_{abb}}="a";
            (20,15)*+{\int^{b} P_{a'bb}}="b";
            (40,0)*+{\int^{b} P_{a''bb}}="c";
            {\ar@/^1pc/^{\int^{b} P_{fbb}} "a" ; "b"};
            {\ar@/^1pc/^{\int^{b} P_{f'bb}} "b" ; "c"};
            {\ar_{\int^{b} P_{(f'f)bb}} "a" ; "c"};
            {\ar@{=>}^{\phi_{f',f}} (18,9.5) ; (18,5.5)};
        \endxy
    \]
induced by the $2$-dimensional universal property of $\int^{b} P_{abb}$, upon being whiskered by $j^{a}_b$, yields the invertible pasting diagram below.
    \[
        \xy
            (0,0)*+{P_{abb}}="a";
            (60,0)*+{\int^{b} P_{abb}}="b";
            (60,-50)*+{\int^{b} P_{a''bb}}="c";
            (0,-50)*+{\int^{b} P_{abb}}="d";
            (30,-50)*+{\int^{b} P_{a'bb}}="e";
            (15,-25)*+{P_{a'bb}}="f";
            (45,-25)*+{P_{a''bb}}="g";
            {\ar^{j^a_b} "a" ; "b"};
            {\ar^{\int^{b}_{(f'f)bb}} "b" ; "c"};
            {\ar_{j^a_b} "a" ; "d"};
            {\ar_{\int^{b} P_{fbb}} "d" ; "e"};
            {\ar_{\int^{b} P_{f'bb}} "e" ; "c"};
            {\ar^{P_{(f'f)bb}} "a" ; "g"};
            {\ar|{P_{fbb}} "a" ; "f"};
            {\ar|{P_{f'bb}} "f" ; "g"};
            {\ar|{j^{a'}_b} "f" ; "e"};
            {\ar|{j^{a''}_b} "g" ; "c"};
            {\ar@{=>}^{j^{f'f}_b} (45,-10.5) ; (45,-14.5)};
            {\ar@{=>}^{j^f_b} (10,-32) ; (10,-36)};
            {\ar@{=>}^{j^{f'}_b} (37.5,-35.5) ; (37.5,-39.5)};
            (20,-17.5)*+{\cong};
        \endxy
    \]
 The unlabeled isomorphism is the composite coherence cell for $P_{-bb}$ consisting of $P(1_{f'f},l_1,l_1)$ and $\phi^P_{f'bb,fbb}$. Similarly, the inverse of the invertible $2$-cell
    \[
        \xy
            (0,0)*+{\int^{b} P_{abb}}="a";
            (30,0)*+{\int^{b} P_{abb}}="b";
            {\ar@/^2pc/^{\int^{b} P_{abb}} "a" ; "b"};
            {\ar@/_2pc/_{id} "a" ; "b"};
            {\ar@{=>}^{\phi^{\int^{b} P_{-bb}}_a} (10,2) ; (10,-2)};
        \endxy
    \]
when whiskered by $j^{a}_b$, gives the invertible pasting diagram
    \[
        \xy
            (0,0)*+{P_{abb}}="a";
            (30,0)*+{\int^{b} P_{abb}}="b";
            (30,-25)*+{\int^{b} P_{abb}}="c";
            (0,-25)*+{\int^{b} P_{abb}}="d";
            (15,-12.5)*+{P_{abb}}="e";
            (20,-7)*+{\cong};
            {\ar@{=>}^{j^{id_a}_b} (7,-14.5) ; (7,-18.5)};
            {\ar^{j^a_b} "a" ; "b"};
            {\ar^{id} "b" ; "c"}:
            {\ar_{j^a_b} "a" ; "d"};
            {\ar_{\int^{b} P_{abb}} "d" ; "c"};
            {\ar|{P_{abb}} "a" ; "e"};
            {\ar|{j^a_b} "e" ; "c"};
        \endxy
    \]
where the unlabeled isomorphism is the composite coherence cell consisting of $l_{{j^a}_b}$, $\left(r_{j^a_b}\right)^{-1}$, and $\phi^{P_{-bb}}_a$.

\begin{lemma}\label{coyoneda}
Let $\m{A}$ be a bicategory. There is a pseudofunctor
\[
    I = \int^{a} -(a) \times \m{A}(-,a) \colon \bf{Bicat}(\m{A}^{op},\bf{Cat}) \rightarrow \bf{Bicat}(\m{A}^{op},\bf{Cat}).
\]
\end{lemma}
\begin{proof}
Since $\bf{Cat}$ is bicocomplete the bicodescent object
    \[
        I(F) = \int^{a} F(a) \times \m{A}(-,a)
    \]
exists for each pseudofunctor $F \colon \m{A}^{op} \rightarrow \bf{Cat}$. Given a pseudonatural transformation $\gamma \colon F \Rightarrow G$, we can define another pseudonatural transformation
    \[
        \gamma \times 1_{\m{A}(-,-)} \colon F \times \m{A}(-,-) \Rightarrow G \times \m{A}(-,-).
    \]
Since $I(F)$ and $I(G)$ are bicodescent object then we also have extrapseudonatural transformations $i^F \colon F \times \m{A}(-,-) \carrow I(F)$, $i^G \colon G \times \m{A}(-,-) \carrow I(G)$, and so the composite of $i^G$ and $\gamma \times 1_{\m{A}(-,-)}$, in the manner of Lemma \ref{comp1}, induces a pseudonatural transformation $I(\gamma) \colon I(F) \Rightarrow I(G)$ via the universal property of $i^F$. This also means there are invertible modifications
    \[
        \xy
            (0,0)*+{Fa \times \m{A}(-,a)}="a";
            (25,0)*+{I(F)}="b";
            (0,-20)*+{Ga \times \m{A}(-,a)}="c";
            (25,-20)*+{I(G)}="d";
            {\ar@{=>}^>>>>>{i^F_a} "a" ; "b"};
            {\ar@{=>}^{I(\gamma)} "b" ; "d"};
            {\ar@{=>}_{\gamma_a \times 1_{\m{A(-,a)}}} "a" ; "c"};
            {\ar@{=>}_>>>>>{i^G_a} "c" ; "d"};
            {\ar@3{->}^{\Gamma_a} (12.5,-8) ; (12.5,-12)};
        \endxy
    \]
satisfying the pasting axiom EB1 of Definition \ref{def:epnatcod}.

The action of $I$ on $2$-cells is described as follows. If $\Sigma \colon \gamma \Rrightarrow \delta$ is a modification then for each $a \in \m{A}$ there is a natural transformation $\Sigma_a \colon \gamma_a \Rightarrow \delta_a$, giving rise to a modification $\Sigma_a \times 1 \colon \gamma_a \times 1 \Rightarrow \delta_a \times 1$. (Note that in the following diagram we switch the style of arrow.) The composite modification
    \[
        \xy
            (0,0)*+{Fa \times \m{A}(-,a)}="a";
            (50,0)*+{I(F)}="b";
            (0,-40)*+{I(F)}="c";
            (50,-40)*+{I(G)}="d";
            (25,-20)*+{Ga \times \m{A}(-,a)}="e";
            {\ar@/^/^{i^F_a} "a" ; "b"};
            {\ar^{I(\gamma)} "b" ; "d"};
            {\ar@/_/_{i^F_a} "a" ; "c"};
            {\ar_{I(\delta)} "c" ; "d"};
            {\ar@/^2pc/|{\gamma_a \times 1_{\m{A}(-,a)}} "a" ; "e"};
            {\ar@/_2pc/|{\delta_a \times 1_{\m{A}(-,a)}} "a" ; "e"};
            {\ar|{i^G_a} "e" ; "d"};
           {\ar@{=>}^{\Gamma_a} (37.5,-8) ; (37.5,-12)};
           {\ar@{=>}^{\Delta^{-1}_a} (12.5,-28) ; (12.5,-32)};
           {\ar@{=>}^{\Sigma_a \times 1} (13.5,-8) ; (10.5,-11)};
        \endxy
    \]
satisfies the requirements of axiom EB2, yielding a unique $2$-cell $I(\Delta) \colon I(\gamma) \Rrightarrow I(\delta)$. The action of $I$ on $2$-cells preserves the strict composition of modifications due to the uniqueness property inherent in the universal property. It remains to describe the data for the pseudofunctor on $1$-cell composition and check the appropriate axioms, however this clearly follows from similar arguments to the above.
\end{proof}

\section{Fubini for codescent objects}

This section makes use of the previous definitions and technical lemmas in order to prove a bicategorical analogue of the Fubini theorem for coends. Similar results have been established via a different approach \cite{Nun16}.

\begin{proposition}\label{fubini1}
Let $P \colon \m{A}^{op} \times \m{B}^{op} \times \m{A} \times \m{B} \rightarrow \m{C}$ be a pseudofunctor and assume that the bicodescent objects
    \[
        j^{a'a} \colon P(a',-,a,-) \overset{\cdot\cdot}{\Rightarrow} \int^{b} P(a',b,a,b)
    \]
and
    \[
        i \colon P \carrow \int^{a,b} P(a,b,a,b)
    \]
exist, where $(a',a) \in \m{A}^{op} \times \m{A}$. Then there is a $1$-cell
    \[
        \sigma \colon \int^{a} \int^{b} P(a,b,a,b) \rightarrow \int^{a,b} P(a,b,a,b)
    \]
if the left side exists.
\end{proposition}

\begin{proof}
 Suppose that the bicodescent object $k \colon \int^{b} P_{-b-b} \carrow \int^{a} \int^{b} P_{abab}$ exists. By Lemma \ref{fixed}, fixing $a \in \m{A}$ results in an extrapseudonatural transformation $i_{a-} \colon P_{a-a-} \carrow \int^{a,b} P_{abab}$ yielding a family of $1$-cells $\phi^{a} \colon \int^{b} P_{abab} \rightarrow \int^{a,b} P_{abab}$ along with corresponding families of invertible $2$-cells
    \[
        \xy
            (0,0)*+{P_{abab}}="a";
            (20,-10)*+{\int^{a,b} P_{abab}}="b";
            (20,10)*+{\int^{b} P_{abab}}="c";
            {\ar_{i_{ab}} "a" ; "b"};
            {\ar^{j^{aa}_b} "a" ; "c"};
            {\ar^{\phi^{a}} "c" ; "b"};
            {\ar@{=>}^{\Phi^a_b} (10,2) ; (10,-2)};
        \endxy
    \]
in $\m{C}$, satisfying the usual axioms. To induce $\sigma$ as in the statement of the theorem we now need find invertible $2$-cells
    \[
        \xy
            (0,0)*+{\int^{b} P_{a'bab}}="a";
            (25,0)*+{\int^{b} P_{abab}}="b";
            (0,-20)*+{\int^{b} P_{a'ba'b}}="c";
            (25,-20)*+{\int^{a,b} P_{abab}}="d";
            {\ar^{\int^{b} P_{fbab}} "a" ; "b"};
            {\ar^{\phi^{a}} "b" ; "d"};
            {\ar_{\int^{b} P_{a'bfb}} "a" ; "c"};
            {\ar_{\phi^{a'}} "c" ; "d"};
            {\ar@{=>}^{\phi^{f}} (12.5,-8) ; (12.5,-12)};
        \endxy
    \]
and show that there is an extrapseudonatural transformation $\phi^{-} \colon \int^{b} P_{-b-b} \carrow \int^{a,b} P_{abab}$.

To find the $\phi^f$ we will use the $2$-dimensional universal properties of the bicodescent objects $\int^{b} P_{a'bab}$. For each $b \in \m{B}$ we have an invertible $2$-cell
    \[
        \xy
            (0,0)*+{P_{a'bab}}="a";
            (25,20)*+{\int^{b} P_{a'bab}}="b";
            (50,20)*+{\int^{b} P_{abab}}="c";
            (75,0)*+{\int^{a,b} P_{abab}}="d";
            (25,-20)*+{\int^{b} P_{a'bab}}="e";
            (50,-20)*+{\int^{b} P_{a'ba'b}}="f";
            (37.5,7.5)*+{P_{abab}}="g";
            (37.5,-7.5)*+{P_{a'ba'b}}="h";
            {\ar^{j^{a'a}_b} "a" ; "b"};
            {\ar^{\int^{b} P_{fbab}} "b" ; "c"};
            {\ar^{\phi^{a}} "c" ; "d"};
            {\ar_{j^{a'a}_b} "a" ; "e"};
            {\ar_{\int^{b} P_{a'bfb}} "e" ; "f"};
            {\ar_{\phi^{a'}} "f" ; "d"};
            {\ar|{P_{fbab}} "a" ; "g"};
            {\ar|{P_{a'bfb}} "a" ; "h"};
            {\ar|{j^{aa}_b} "g" ; "c"};
            {\ar|{j^{a'a'}_b} "h" ; "f"};
            {\ar|{i_{ab}} "g" ; "d"};
            {\ar|{i_{a'b}} "h" ; "d"};
            {\ar@{=>}^{\left(j^{fa}_b\right)^{-1}} (25,15) ; (25,11)};
            {\ar@{=>}^{j^{a'f}_b} (25,-11) ; (25,-15)};
            {\ar@{=>}^{i_{fb}} (37.5,2) ; (37.5,-2)};
            {\ar@{=>}^{\Phi^a_b} (50,12) ; (50,8)};
            {\ar@{=>}^{\left(\Phi^{a'}_b\right)^{-1}} (50,-8) ; (50,-12)};
        \endxy
    \]
which satisfies the pasting conditions of axiom EB2. This is seen by pasting these $2$-cells with $j^{a'a}_g$ for some $g \colon b \rightarrow b'$ in $\m{B}$ and using properties of the $j_{b}^{fa}$, properties of the $\Phi^{a}_{b}$, axiom EP1, and axiom EP6, before using similar applications of these in the reverse order. Hence there is a unique $2$-cell $\phi^{f}$ as required, which satisfies appropriate pasting conditions, namely that the whiskering of $\phi^f$ by $j^{a'a}_b$ yields the composite $2$-cell displayed above.

We must now check that these $\phi^f$ satisfy the axioms of an extrapseudonatural transformation. As the codomain of the $1$-cells is an object of $\m{C}$ then some of the axioms again become redundant, namely EP2-4, and EP7. For EP1 we whisker each of the diagrams by $j^{a''a}_b$. On one side we get an instance of the above $2$-cell for $f'f$, while on the other we have to use the properties of the coherence cells of $\int^{b}P_{-b-b}$, in the manner described following Lemma \ref{fixed}. To equate the two pasting diagrams is a case of using the pseudonaturality of $j_b$, extrapseudonaturality of $i$, and instances of the above composite $2$-cell for both $f$ and $f'$. The uniqueness in the $2$-dimensional universal property of $\int^{b} P_{a''bab}$ is used to show that EP1 then holds. For EP5, most of the $2$-cells in $\phi^{id_a} \ast 1_{j^{aa}_b}$ are identities, leaving $\Phi^{a}_b$ to cancel with itself, before again using axiom EB2 to show the equality. Axiom EP6 is simple to check.

Since $\phi \colon \int^{b} P_{-b-b} \carrow \int^{a,b} P_{abab}$ is then an extrapseudonatural transformation there exists an invertible $2$-cell
    \[
        \xy
            (0,0)*+{\int^{b} P_{abab}}="a";
            (25,-10)*+{\int^{a,b} P_{abab}}="b";
            (25,10)*+{\int^{a} \int^{b} P_{abab}}="c";
            {\ar_{\phi^a} "a" ; "b"};
            {\ar^{k_a} "a" ; "c"};
            {\ar^{\sigma} "c" ; "b"};
            {\ar@{=>}^{\Sigma_a} (15,2) ; (15,-2)};
        \endxy
    \]
for each $a \in \m{A}$.
\end{proof}

\begin{lemma}\label{fublemma}
Let $P \colon \m{A}^{op} \times \m{B}^{op} \times \m{A} \times \m{B} \rightarrow \m{C}$ be a pseudofunctor and assume that for each fixed pair $(a',a) \in \m{A}^{op} \times \m{A}$ the bicodescent object
    \[
       j^{a'a} \colon P(a',-,a,-) \overset{\cdot\cdot}{\Rightarrow} \int^{b} P(a',b,a,b)
    \]
exists. Similarly suppose that the bicodescent object
    \[
        \theta \colon \int^{a,b} P(a,b,a,b) \rightarrow \int^{a} \int^{b} P(a,b,a,b)
    \]
exists. For each fixed $b \in \m{B}$, considering $j_b$ as a pseudonatural transformation $j \colon P_{-b-b} \Rightarrow \int^{b} P_{-b-b}$, the composite of $k$ and $j$, as in Lemma \ref{comp1} satisfies the compatibility condition of Lemma \ref{compatibility}.
\end{lemma}
\begin{proof}
The proof relies on the equality of certain pasting diagrams, as in Lemma \ref{compatibility}. The easiest way to prove this equality is to show that one of the pasting diagrams acts as an inverse for the other. The steps required depend on how the $j_b$ interact with the $j^{a'a}$, as specified by Lemma \ref{psfunctor}.
\end{proof}

\begin{proposition}\label{fubini2}
Let $P \colon \m{A}^{op} \times \m{B}^{op} \times \m{A} \times \m{B} \rightarrow \m{C}$ be a pseudofunctor and assume that the bicodescent objects
    \[
       j^{a'a} \colon P(a',-,a,-) \overset{\cdot\cdot}{\Rightarrow} \int^{b} P(a',b,a,b)
    \]
and
    \[
        k \colon \int^{b} P(-,b,-,b) \carrow \int^{a} \int^{b} P(a,b,a,b)
    \]
exist, where $(a',a) \in \m{A}^{op} \times \m{A}$. Then there is a $1$-cell
    \[
        \theta \colon \int^{a,b} P(a,b,a,b) \rightarrow \int^{a} \int^{b} P(a,b,a,b)
    \]
if the left side exists.
\end{proposition}
\begin{proof}
Suppose that the bicodescent object $i \colon P \carrow \int^{a,b} P_{abab}$ exists. By Lemma \ref{comp1}, the composite of $j$ and $k$ is extrapseudonatural in $a$. This composite is also extrapseudonatural in $b$, following from the extrapseudonaturality of $j$, simply by whiskering diagrams with the $1$-cells of $k$. By the previous lemma, the composite of $j$ and $k$ is then extrapseudonatural in $(a,b)$, so there exists an invertible $2$-cell
    \[
        \xy
            (0,0)*+{P_{abab}}="a";
            (30,0)*+{\int^{a,b} P_{abab}}="b";
            (0,-20)*+{\int^{b} P_{abab}}="c";
            (30,-20)*+{\int^{a} \int^{b} P_{abab}}="d";
            {\ar^{i_{ab}} "a" ; "b"};
            {\ar_{j^{aa}_b} "a" ; "c"};
            {\ar_{k_a} "c" ; "d"};
            {\ar^{\theta} "b" ; "d"};
            {\ar@{=>}^{\Theta_{ab}} (13,-8) ; (13,-12)};
        \endxy
    \]
for each $(a,b) \in \m{A} \times \m{B}$.
\end{proof}
\begin{theorem}\label{fubini}
Under the conditions of Proposition \ref{fubini1} and Proposition \ref{fubini2} there is an adjoint equivalence
    \[
        \int^{a,b} P(a,b,a,b) \simeq \int^{a} \int^{b} P(a,b,a,b).
    \]
\end{theorem}
\begin{proof}
The equivalence is provided by the $1$-cells and invertible $2$-cells induced in the previous theorems. We require isomorphisms
    \[
        \sigma \cdot \theta \cong id, \theta \cdot \sigma \cong id
    \]
before showing that they satisfy appropriate axioms. For the first isomorphism, we can show that the invertible $2$-cells
    \[
        \xy
            (0,0)*+{P_{abab}}="a";
            (30,20)*+{\int^{a,b} P_{abab}}="b";
            (60,20)*+{\int^{a}\int^{b} P_{abab}}="c";
            (20,-20)*+{\int^{a,b} P_{abab}}="d";
            (60,-20)*+{\int^{a,b} P_{abab}}="e";
            (40,0)*+{\int^{b} P_{abab}}="f";
            {\ar@/^1pc/^{i_{ab}} "a" ; "b"};
            {\ar^{\theta} "b" ; "c"};
            {\ar@/^1pc/^{\sigma} "c" ; "e"};
            {\ar@/_1pc/_{i_{ab}} "a" ; "d"};
            {\ar_{id} "d" ; "e"};
            {\ar|{j^{aa}_b} "a" ; "f"};
            {\ar|{k_a} "f" ; "c"};
            {\ar@/^1pc/|{\phi^a} "f" ; "e"};
            {\ar|{i_{ab}} "a" ; "e"};
            {\ar@{=>}^{\Theta_{ab}} (25,12) ; (25,8)};
            {\ar@{=>}^{\Sigma_a} (55,2) ; (55,-2)};
            {\ar@{=>}^{\Phi^a_b} (40,-5) ; (40,-9)};
            {\ar@{=>}^{l_{i_{ab}}^{-1}} (16,-11) ; (16,-15)};
        \endxy
    \]
satisfy the requirements of axiom EB2, giving a unique invertible $2$-cell $\kappa \colon \sigma \cdot \theta \Rightarrow id$ such that $\kappa \ast 1_{i_{ab}}$ is the pasting diagram above.

The second isomorphism requires two steps. The first uses invertible $2$-cells
    \[
        \xy
            (0,0)*+{P_{abab}}="a";
            (25,20)*+{\int^{b} P_{abab}}="b";
            (50,20)*+{\int^{a,b} P_{abab}}="c";
            (75,0)*+{\int^{a,b} P_{abab}}="d";
            (25,-20)*+{\int^{b} P_{abab}}="e";
            (50,-20)*+{\int^{a,b} P_{abab}}="f";
            {\ar^{j^{aa}_b} "a" ; "b"};
            {\ar^{\phi^a} "b" ; "c"};
            {\ar^{\theta} "c" ; "d"};
            {\ar_{j^{aa}_b} "a" ; "e"};
            {\ar_{k_a} "e" ; "f"};
            {\ar_{id} "f" ; "d"};
            {\ar@/^1pc/|{k_a} "e" ; "d"};
            {\ar@/_1pc/|{i_{ab}} "a" ; "c"};
            {\ar@{=>}^{\Phi^a_b} (25,14) ; (25,10)};
            {\ar@{=>}^{l_{i_{k_a}}^{-1}} (50,-10) ; (50,-14)};
            {\ar@{=>}^{\Theta_{ab}} (37.5,2) ; (37.5,-2)};
        \endxy
    \]
satisfying the requirements of axiom EB2 to give unique invertible $2$-cells $\Omega_a \colon \theta \cdot \phi^a \Rightarrow id \cdot k_a$ such that $\Omega_a \ast 1_{j^{aa}_b}$ is the pasting diagram above. The second step uses invertible $2$-cells
    \[
        \xy
            (5,5)*+{\int^{b} P_{abab}}="a";
            (30,20)*+{\int^{a}\int^{b} P_{abab}}="b";
            (60,20)*+{\int^{a,b} P_{abab}}="c";
            (20,-10)*+{\int^{a}\int^{b} P_{abab}}="d";
            (60,-10)*+{\int^{a}\int^{b} P_{abab}}="e";
            {\ar^{k_a} "a" ; "b"};
            {\ar^{\sigma} "b" ; "c"};
            {\ar@/^1pc/^{\theta} "c" ; "e"};
            {\ar_{k_a} "a" ; "d"};
            {\ar_{id} "d" ; "e"};
            {\ar@/_2pc/|{\phi^a} "a" ; "c"};
            {\ar@{=>}^{\Sigma_a} (30,12) ; (30,8)};
            {\ar@{=>}^{\Omega_a} (40,2) ; (40,-2)};
        \endxy
    \]
which again satisfy the requirements of axiom EB2, in order to give unique invertible $2$-cells $\lambda \colon \theta \cdot \sigma \Rightarrow id$ such that $\lambda \ast 1_{k_a}$ is the pasting diagram above. To apply EB2 in this instance requires that, for some $f \colon a \rightarrow a'$, the pasting of the two instances of the previous diagram, for $a$ and $a'$, with $k_f$ are equal. We show that they are equal by using the universal property of $\int^{b} P_{abab}$, whiskering the diagrams with $j^{a'a}_b$ gives an equality of pasting diagrams and by uniqueness the original diagrams are equal.

Checking that this is then an adjoint equivalence relies again on the axiom EB2. The check here is somewhat simpler than the previous calculations. For each pair $(a,b) \in \m{A} \times \m{B}$ we have invertible $2$-cells
    \[
        \xy
            (0,0)*+{P_{abab}}="a";
            (30,0)*+{\int^{a,b} P_{abab}}="b";
            (90,0)*+{\int^{a} \int^{b} P_{abab}}="c";
            (60,10)*+{\int^{a} \int^{b} P_{abab}}="d";
            (60,-10)*+{\int^{a,b} P_{abab}}="e";
            {\ar^{i_{ab}} "a" ; "b"};
            {\ar@/^6pc/^{\theta} "b" ; "c"};
            {\ar@/_6pc/_{\theta} "b" ; "c"};
            {\ar|{\theta} "b" ; "d"};
            {\ar|{id} "d" ; "c"};
            {\ar|{id} "b" ; "e"};
            {\ar|{\theta} "e" ; "c"};
            {\ar^{\sigma} "d" ; "e"};
            {\ar@{=>}^{l_{\theta}^{-1}} (60,19) ; (60,15)};
            {\ar@{=>}^{r_{\theta}} (60,-15) ; (60,-19)};
            {\ar@{=>}^{\kappa} (50,2) ; (50,-2)};
            {\ar@{=>}^{\lambda^{-1}} (70,2) ; (70,-2)};
        \endxy
    \]
which plainly satisfy the requirements of EB2. We also note that this whiskered pasting diagram is equal to the identity on $\theta \cdot i_{ab}$ and so by uniqueness we find that the composite $2$-cell, when not whiskered by $i_{ab}$, is the identity on $\theta$. A similar argument shows that the other triangle identity also holds, hence the equivalence is in fact an adjoint equivalence.
\end{proof}

\begin{corollary}
Let $P \colon \m{A}^{op} \times \m{B}^{op} \times \m{A} \times \m{B} \rightarrow \m{C}$ be a pseudofunctor and assume that  the bicodescent objects
    \[
       j^{a'a} \colon P(a',-,a,-) \overset{\cdot\cdot}{\Rightarrow} \int^{b} P(a',b,a,b)
    \]
and
    \[
        l^{b'b} \colon P(-,b',-,b) \carrow \int^{a} P(a,b',a,b)
    \]
exist, where $(a',a) \in \m{A}^{op} \times \m{A}$ and $(b',b) \in \m{B}^{op} \times \m{B}$. Then there is an adjoint equivalence
    \[
        \int^{a} \int^{b} P(a,b,a,b) \simeq \int^{b} \int^{a} P(a,b,a,b).
    \]
\end{corollary}

It is simple to see that an object satisfying the axioms of a codescent object \cite{Lac02}, \cite{Str87} also satisfies those of a bicodescent object. A bicodescent object only requires existence of an induced $1$-cell in the $1$-dimensional property whereas a codescent object requires this to also be unique. Hence our Fubini result applies to codescent objects as well as bicodescent objects.

\bibliographystyle{apalike}
\bibliography{epnatbib}

\end{document}